\documentclass{article}
\usepackage{graphicx} 
\usepackage{amssymb}
\usepackage{amsmath}
\usepackage{amsthm}
\usepackage{bm}
\usepackage{amscd}
\usepackage[all]{xy}
\usepackage{stmaryrd}
\usepackage[margin=1in]{geometry}
\usepackage{tikz-cd}
\usepackage{booktabs}
\usepackage{xcolor}
\usepackage{parskip}

\usepackage{xcolor}
\definecolor{Mylinkclr}{HTML}{666666}
\usepackage[pagebackref,colorlinks,linkcolor=Mylinkclr,citecolor=Mylinkclr]{hyperref}

\usepackage{zref-clever}
\zcsetup{cap}

\zcRefTypeSetup{equation}{
    Name-sg=,
    name-sg=,
    Name-pl=,
    name-pl=,
    Name-sg-ab=,
    name-sg-ab=,
    Name-pl-ab=,
    name-pl-ab=
}

\theoremstyle{plain}
\newtheorem{Thm}{{\textbf{Theorem}}}[section]
\AddToHook{env/Thm/begin}{%
   \zcsetup{countertype={Thm=Thm}}}
\zcRefTypeSetup{Thm}{Name-sg=}

\newtheorem{Lem}[Thm]{{\textbf{Lemma}}}
\AddToHook{env/Lem/begin}{%
   \zcsetup{countertype={Thm=Lem}}}
\zcRefTypeSetup{Lem}{Name-sg=}

\newtheorem{Prop}[Thm]{{\textbf{Proposition}}}
\AddToHook{env/Prop/begin}{%
   \zcsetup{countertype={Thm=Prop}}}
\zcRefTypeSetup{Prop}{Name-sg=}

\newtheorem{Cor}[Thm]{{\textbf{Corollary}}}
\AddToHook{env/Cor/begin}{%
   \zcsetup{countertype={Thm=Cor}}}
\zcRefTypeSetup{Cor}{Name-sg=}

\AddToHook{env/Conj/begin}{%
   \zcsetup{countertype={Thm=Conj}}}
\zcRefTypeSetup{Conj}{Name-sg=}

\newtheorem{Def}[Thm]{{\textbf{Definition}}}
\AddToHook{env/Def/begin}{%
   \zcsetup{countertype={Thm=Def}}}
\zcRefTypeSetup{Def}{Name-sg=}

\newtheorem{Rem}[Thm]{{\textbf{Remark}}}
\AddToHook{env/Rem/begin}{%
   \zcsetup{countertype={Thm=Rem}}}
\zcRefTypeSetup{Rem}{Name-sg=}

\AddToHook{env/Ex/begin}{%
   \zcsetup{countertype={Thm=Ex}}}
\zcRefTypeSetup{Ex}{Name-sg=}

\theoremstyle{plain}

\newcommand{\cA}{{\mathcal A}}
\newcommand{\cB}{{\mathcal B}}

\newcommand{\cO}{{\mathcal O}}

\newcommand{\cR}{{\mathcal R}}

\newcommand{\cU}{{\mathcal U}}

\newcommand{\cX}{{\mathcal X}}
\newcommand{\cY}{{\mathcal Y}}

\newcommand{\bA}{{\mathbb A}}

\newcommand{\bC}{{\mathbb C}}

\newcommand{\bP}{{\mathbb P}}
\newcommand{\bQ}{{\mathbb Q}}

\newcommand{\bZ}{{\mathbb Z}}

\newcommand{\bz}{{\bm z}}

\newcommand{\e}{\epsilon}

\DeclareMathOperator{\Supp}{Supp}
\DeclareMathOperator{\Pic}{Pic}

\DeclareMathOperator{\id}{id}
\DeclareMathOperator{\Sym}{Sym}

\DeclareMathOperator{\Aut}{Aut}

\DeclareMathOperator{\Div}{div}

\DeclareMathOperator{\Spec}{Spec}

\DeclareMathOperator{\IC}{IC}

\DeclareMathOperator{\sq}{sq}
\DeclareMathOperator{\orb}{orb}
\DeclareMathOperator{\UConf}{UConf}

\DeclareMathOperator{\D}{def}
\DeclareMathOperator{\rank}{rank}
\DeclareMathOperator{\ab}{ab}
\DeclareMathOperator{\Sing}{Sing}

\newcommand\ra{\rightarrow}
\newcommand{\bsq}{\bm{\sq}}

\newcommand{\Address}{{
  \bigskip
  \footnotesize
  
  Shomrik Bhattacharya, \textsc{Centre for Quantum Mathematics, University of Southern Denmark, Campusvej 55, DK-5230, Odense. M, Denmark}\par\nopagebreak
  \textit{E-mail address}: \texttt{bhattacharya@imada.sdu.dk}
 \vspace{0.5cm}
   
  Yuki Matsubara, \textsc{Centre for Quantum Mathematics, University of Southern Denmark, Campusvej 55, DK-5230, Odense. M, Denmark}\par\nopagebreak
  \textit{E-mail address}: \texttt{matsubara@imada.sdu.dk}

}}

\title{Fundamental groups of complements of wobbly divisors in intersections of two quadrics}
\author{Shomrik Bhattacharya, Yuki Matsubara}
\date{}

\begin{document}
\maketitle

\begin{abstract}
A smooth complete intersection $X\subset\mathbb{P}^{n+2}_{\mathbb{C}}$ of two quadrics admits an interpretation as a moduli space of bundles. For $n=2$, it parametrizes stable rank $2$ parabolic bundles on $\mathbb{P}^1$ with five marked points, while for $n=3$ it parametrizes stable rank $2$ bundles of odd
degree with fixed determinant on a genus $2$ curve. 
For general $n$, it can be interpreted as a moduli space of semistable twisted $\mathrm{Spin}$-bundles.
We study $X\setminus W$, where $W$ is the wobbly locus, consisting of bundles that admit a nonzero nilpotent Higgs field.
We prove that $W$ is an irreducible divisor for $n\geq 3$ and compute $H_1(X\setminus W,\mathbb{Z})$. Using a root-stack description, we identify $\pi_1(X\setminus W)$ with the kernel of a monodromy homomorphism from an orbifold mixed braid group. 
The Reidemeister–Schreier method then yields an explicit presentation in every dimension. 
For $n=2$, where $X$ is a degree $4$ del Pezzo surface and $W$ is the union of its sixteen $(-1)$-curves, we obtain a presentation with $10$ generators and $25$ relators and prove that both numbers are minimal among all finite presentations.
For $n = 3$, we show that $\pi_1(X \setminus W)$ admits a presentation with $4$ generators and $78$ relators.

\end{abstract}

\section{Introduction}
\subsection{Motivation}\label{Motivation}

A vector bundle on a curve is called \emph{very stable} if it admits no nonzero nilpotent Higgs field, and \emph{wobbly} otherwise. 
This notion goes back to Laumon~\cite{Lau}. 
Donagi and Pantev studied the wobbly locus in detail in the moduli space of rank $2$ parabolic bundles on $\mathbb{P}^1_{\mathbb{C}}$ with five marked points~\cite{DP}. 
In the geometric Langlands correspondence, Hecke eigensheaves are expected to be obtained as $\IC$-extensions of local systems on the complement of the wobbly locus, the very stable locus.
This motivates studying the topology of the very stable locus, and in particular its fundamental group, which governs the local systems on this locus.

Smooth intersections of two quadrics provide a family of moduli spaces in
which the wobbly locus can be described explicitly. Let
$X = Q_0 \cap Q_1 \subset \mathbb{P}^{n+2}_{\mathbb{C}}$ be a smooth $n$-dimensional intersection of two quadrics (see \zcref{Geometric_setup} for the
precise setup). For $n = 2$ and $3$, $X$ is a classical moduli space of vector
bundles:
\begin{itemize}
  \item if $n = 2$, then $X$ is isomorphic to the moduli space of rank $2$
  stable parabolic bundles of weight $w = (\tfrac12, \dots, \tfrac12)$ over
  $(\mathbb{P}^1_{\mathbb{C}}, \lambda_0 + \dots + \lambda_4)$, which is a
  del Pezzo surface of degree $4$ (cf.~\cite[Chapter~3]{DP}),
  \item if $n = 3$, then $X$ is isomorphic to the moduli space of rank $2$
  stable vector bundles of odd degree with fixed determinant over a genus
  $2$ curve~\cite[Theorem~1]{N68}.
\end{itemize}
For general $n$, Benedetti, H\"oring and Liu interpret $X$ as
a moduli space of semistable twisted $\mathrm{Spin}$-bundles \cite[Sections~1.2--1.3]{BHL25}. Under these identifications, the
cotangent bundle $T^*X$ carries a Hitchin integrable system,
studied by Hitchin~\cite{Hit25} and in~\cite{BHL25}. 
Hitchin gave a simple criterion for very stable bundles in terms of the roots of an associated polynomial
(\cite[Proposition 1]{Hit25}).

The resulting wobbly locus $W \subset X$ is a divisor. 
For $n = 3$, this follows from the work of Pal and Pauly, who also computed the class of $W$ \cite[Theorem 1.1, Section 5.1.2]{PP21}.
For $n = 2$, Donagi and Pantev identified it with the union of the sixteen $(-1)$-curves~\cite[Theorem~C]{DP}. 
In this paper, we study the topology of the very stable locus, the complement of the wobbly locus $X \setminus W$ for every $n$. Our main object is
its fundamental group $\pi_1(X \setminus W)$. The results are summarized in
Theorem \zcref{Summary_theorem}.

\subsection{Geometric setup}\label{Geometric_setup}

We use homogeneous coordinates $[x_0 : x_1 : \cdots : x_{n+2}]$ on $\mathbb{P}^{n+2}_{\mathbb{C}}$ with $n \geq 1$.
Consider the $n$-dimensional subvariety
\begin{equation*}
    X := Q_0 \cap Q_1 \subset \mathbb{P}^{n+2}_{\mathbb{C}},
\end{equation*}
where the quadrics $Q_0$ and $Q_1$ are defined by
\begin{equation*}
      Q_0 := \left\{ \sum_{j=0}^{n+2} x_j^2 = 0 \right\},
  \qquad
  Q_1 := \left\{ \sum_{j=0}^{n+2} \lambda_j x_j^2 = 0 \right\},
\end{equation*}
with pairwise distinct $\lambda_0, \dots, \lambda_{n+2} \in \bC$.
The variety $X$ is connected and smooth of complex dimension $n$ \cite[Proposition 2.1]{Reid1972}.

Following \cite[Proposition 1]{Hit25}, we consider the map
\begin{equation*}
    f \colon X \to \mathrm{Sym}^n(\mathbb{P}_{\bC}^1) \cong \mathbb{P}_{\bC}^n 
\end{equation*}
given by the following commutative diagram

\[
  \begin{tikzcd}[row sep=3em, column sep=4em]
    X=Q_0\cap Q_1  \arrow[d, "f"'] \arrow[r, hook] &
    \mathbb{P}^{n+2}_{\mathbb{C}} \arrow[d, "\bsq"] \\
    \mathrm{Sym}^n(\mathbb{P}_{\bC}^1)\simeq \mathbb{P}_{\bC}^n \arrow[r, hook] &
    \mathbb{P}^{n+2}_{\mathbb{C}}.
  \end{tikzcd}
\]
Here $\bsq$ denotes the coordinate-wise squaring map
$[x_0:\cdots:x_{n+2}]\mapsto [x_0^2:\cdots:x_{n+2}^2]$, and $f := \bsq|_{X}$.

Since $(\bZ/2\bZ)^{n+3}$ acts on $\bP^{n+2}_{\bC}$ by 
\begin{equation*}
    (\e_0, \dots, \e_{n+2}). [x_0 : \cdots : x_{n+2}] = [\e_0 x_0 : \cdots : \e_{n+2} x_{n+2}], \ \e_j \in \{1, -1 \},
\end{equation*}
the map $\bsq$ is a branched Galois cover with Galois group 
\begin{equation*}
    G := (\bZ/2\bZ)^{n+3}/\langle (-1, \dots, -1) \rangle \simeq (\bZ/2\bZ)^{n+2}.
\end{equation*}
It is branched along the hyperplanes $\{ y_j = 0 \}$ for $j = 0, \dots, n+2$, where $[y_0 : \cdots : y_{n+2}]$ are the homogeneous coordinates in the target $\bP_{\bC}^{n+2}$.
With these coordinates, we can describe the image of $f$ as
\begin{equation*}
    f(X) = \left\{\sum y_j = 0, \ \sum \lambda_j y_j = 0 \right\} \simeq \bP^n_{\bC},
\end{equation*}
and we have $X/G \simeq \bP^n_{\bC}$.

One convenient description of $f$ is the following.
For $x \in X$ and $z \in \bC$, set 
\begin{equation}\label{poly_p}
    p_x(z) := \sum_{j = 0}^{n+2} x_j^2 \prod_{\ell \neq j} (z - \lambda_{\ell}).
\end{equation}
The equations of $X = Q_0 \cap Q_1$ imply $\deg(p_x) \leq n$, and 
\begin{equation}\label{polynomial_p_x_lambda}
    p_x(\lambda_j) = x_j^2\prod_{\ell \neq j} (\lambda_j - \lambda_{\ell}),\  \ x_j^2 = \frac{p_x(\lambda_j)}{\prod_{\ell \neq j} (\lambda_j - \lambda_{\ell})}.
\end{equation}
Therefore, after homogenizing to degree $n$, the zero divisor $\Div(p_x) = q_1 + \cdots +q_n$ gives us the map from $X$ to $|\cO_{\bP^1}(n)| \simeq \Sym^n(\bP^1_{\bC}) \simeq \bP^n_{\bC}$, and this map is identified with $f = \bsq|_X$.

For each $j = 0, 1,  \dots, n+2$, define
\begin{equation*}
    B_j := \{D \in |\cO_{\bP^1}(n)|  \ | \ \lambda_j \in \Supp(D)\} \subset \bP^n_{\bC}, \ \ R_j := X \cap \{x_j = 0 \}.
\end{equation*}
By definition of $f \colon X \ra \bP^n_{\bC}$, the union $B := \cup_{j = 0}^{n+2} B_j$ is the branch locus of $f$, and the union $R := \cup_{j= 0}^{n+2} R_j$ is the ramification locus of $f$. 

\begin{Prop}[{\cite[Proposition 1]{Hit25}}]\label{Hitchin_prop}
    A point $x = [x_0: \cdots: x_{n+2}]$ on the intersection of quadrics $X = Q_0 \cap Q_1$ is very stable with respect to the integrable system (\zcref{Motivation}) if and only if the polynomial $p_x(z)$ has distinct roots (including $z = \infty$).
\end{Prop}
Let $\Delta \subset \Sym^n(\bP^1_{\bC}) \simeq \bP^n_{\bC}$ be the discriminant locus, and set
\begin{equation*}
    W:= f^{-1}(\Delta)_{red}, \ \ U := \bP^n_{\bC} \setminus \Delta, \ \ Y := X \setminus W.
\end{equation*}
Then Proposition \zcref{Hitchin_prop} implies that $W$ is the {\it wobbly locus} and $Y$ is the {\it{very stable locus}} with respect to the corresponding integrable system (\zcref{Motivation}).
We are interested in computing $\pi_1(Y) = \pi_1(X \setminus W)$.

\subsection{Results}
Our computation of $\pi_1(X \setminus W)$ is based on the following description in terms of braid groups.
Let 
\begin{equation*}
    SB_{n+3, n} := \pi_1(\UConf_n(\bP^1_{\bC} \setminus \{ \lambda_0, \dots, \lambda_{n+2}\})) \simeq \pi_1(\bP^n_{\bC} \setminus (\Delta \cup B))
\end{equation*}
be the spherical mixed braid group, with meridians $a_j$ and half-twists $\sigma_i$.
Here, $\UConf_n(S)$ is the unordered configuration space of $n$ distinct points on $S$.
Let $ \langle \langle a_0^2, \dots, a_{n+2}^2 \rangle \rangle_{\cB}$ be the normal closure of $\{a_0^2, \dots, a_{n+2}^2  \}$ within $SB_{n+3, n}$, and $\eta_j \in G$ be the class of the sign vector whose $j$th entry is $-1$ and whose other entries are $1$.
In this notation, we can describe $\pi_1(X \setminus W)$ as follows.
\begin{Thm}[Corollary \zcref{ker_description_of_pi_1}]
    For every $n \geq 1$,
    \begin{equation*}
        \pi_1(X \setminus W) \simeq \ker \left(SB_{n+3, n}/ \langle \langle a_0^2, \dots, a_{n+2}^2 \rangle \rangle_{\cB} \xrightarrow{\tilde{\mu}}G\right),
    \end{equation*}
    where $\tilde{\mu}$ is the surjective monodromy homomorphism satisfying $\tilde{\mu}(a_j) = \eta_j$ and $\tilde{\mu}(\sigma_i) = 1_G$.
\end{Thm}
This description follows from a root-stack interpretation of the branched covering $f$, and allows us to apply the Reidemeister--Schreier method to obtain an explicit presentation.

We summarize our results in the following theorem.
\begin{Thm}\label{Summary_theorem}
    \begin{enumerate}
        \item[(1)] For $n = 1$, $X$ is an elliptic curve in $\bP^3_{\bC}$, $W = \emptyset$, and
                    \begin{equation*}
                      \pi_1(X \setminus W) = \pi_1(X) \simeq H_1(X, \bZ) \simeq \bZ^2.
                    \end{equation*}
        \item[(2)] For $n = 2$, $X$ is a degree $4$ del Pezzo surface $dP_5$ in $\bP^4_{\bC}$, and $W$ is the union of $16$ $(-1)$-curves.
                   The Reidemeister–Schreier presentation of $\pi_1(X \setminus W)$ constructed in \zcref{Section_Reidemister_Schreier} has $65$ generators and $240$ formal relators.
                   It admits a presentation with $10$ generators and $25$ relators, and both numbers are minimal among all finite presentations.
                   Its abelianization $H_1(X\setminus W, \bZ)$ is isomorphic to $\bZ^{10}$, and its second group homology $H_2(\pi_1(X \setminus W), \bZ)$ is isomorphic to $\bZ^{25}$.
        \item[(3)] For $n \geq 3$, $W$ is irreducible of class $4(n-1)H$ with $H = c_1(\cO_X(1))$, and has projective degree $16(n-1)$ in $X$.
                   The Reidemeister–Schreier presentation of $\pi_1(X \setminus W)$ has $2^{n+3}n +1$ generators and $2^{n+2}(2n^2 + 2n + 3)$ formal relators.
                   Its abelianization $H_1(X\setminus W, \bZ)$ is isomorphic to $\bZ/4(n-1)\bZ$.
                   For $n = 3$, $\pi_1(X \setminus W)$ admits a presentation with $4$ generators and $78$ relators.
     \end{enumerate}
\end{Thm}

\subsection{Outline of the paper}
We briefly outline the contents of this paper.
In \zcref{irreducibility_homology_section}, we prove the irreducibility of $W$ for $n \geq 3$, and compute the first homology group of $X \setminus W$.
In \zcref{fundamental_group_section}, by using a root stack and quotient stack description, we identify the fundamental group $\pi_1(X \setminus W)$ with a subgroup of an orbifold mixed braid group.
In \zcref{Section_Reidemister_Schreier}, we use the Reidemeister--Schreier method to obtain the explicit description of $\pi_1(X \setminus W)$.
In \zcref{example_section}, we apply our results to the case $n = 1, 2, 3$.

\subsection*{Acknowledgements}
We would like to thank our Ph.D. supervisor, Vivek Shende, for suggesting this project and for his warm encouragement.
The work presented in this paper is supported by VILLUM FONDEN, VILLUM Investigator grant 37814, Novo Nordisk Foundation grant NNF20OC0066298, and academist crowdfunding.

\section{Irreducibility of $W$ and the First Homology Group of $X \setminus W$}\label{irreducibility_homology_section}
In this section, we prove the irreducibility of $W$ for $n \geq 3$, and compute the first homology group of $X \setminus W$.
\subsection{Geometric descriptions of the wobbly locus}
Define
\begin{equation*}
    Z_n := \left\{ [x_0: x_1 : \cdots : x_{n+2}] \in \bP^{n+2}_{\bC}\ \middle\vert \ \sum_j \lambda_j^r x_j^2 = 0, \ r = 0, 1, 2, 3 \right\} \subset \bP^{n+2}_{\bC}.
\end{equation*}
That is, if we denote
\begin{equation*}
     Q_r := \left\{ [x_0: x_1 : \cdots : x_{n+2}] \in \bP^{n+2}_{\bC}\ \middle\vert \ \sum_{j=0}^{n+2} \lambda_j^r x_j^2 = 0 \right\}
\end{equation*}
for $r= 0, 1, 2, 3$, then $Z_n = Q_0 \cap Q_1 \cap Q_2 \cap Q_3 \subset \bP^{n+2}_{\bC}$.

\begin{Lem}\label{Z_irreducibility_lemma}
\begin{enumerate}
    \item[(1)]If $n = 1$, then $Z_1$ is empty.
    \item[(2)]If $n = 2$, then the scheme $Z_2$ consists of $16$ distinct reduced points. 
    \item[(3)]If $n \geq 3$, then the scheme $Z_n$ is a smooth, irreducible complete intersection of four quadrics of dimension $n-2$.
\end{enumerate}
\end{Lem}
\begin{proof}
    For $r = 0, 1, 2, 3$, set 
    \begin{equation*}
        F_r(x) := \sum_{j = 0}^{n+2} \lambda_j^r x_j^2.
    \end{equation*}
    Thus, $Z_n = V(F_0, F_1, F_2, F_3) \subset \bP^{n+2}_{\bC}$.
    Define the projective linear subspace
    \begin{equation*}
        \Lambda_n := \left\{ [y_0: y_1 : \cdots : y_{n+2}] \in \bP^{n+2}_{\bC}\ \middle\vert \ \sum_j \lambda_j^r y_j = 0, \ r = 0, 1, 2, 3 \right\} \subset \bP^{n+2}_{\bC}.
    \end{equation*}
    By construction, we have a scheme-theoretic equality $Z_n = \bsq^{-1}(\Lambda_n)$.
    The coefficient matrix of the four linear equations defining $\Lambda_n$ is
    \begin{equation*}
        V = \begin{pmatrix}
          1 & 1 & \cdots & 1\\
          \lambda_0 & \lambda_1 & \cdots & \lambda_{n+2}\\
          \lambda_0^2 & \lambda_1^2 & \cdots & \lambda_{n+2}^2\\
          \lambda_0^3 & \lambda_1^3 & \cdots & \lambda_{n+2}^3
        \end{pmatrix}.
    \end{equation*}
    Any four columns of $V$ form a Vandermonde matrix, whose determinant is $\prod_{1\leq k < \ell \leq 4}(\lambda_{j_{\ell}}-\lambda_{j_k})$.
    Since $\lambda_j$ are pairwise distinct, this determinant is nonzero.
    Hence, $V$ has rank $4$.

    Suppose first $n=1$.
    Then $V$ is an invertible $4 \times 4$ matrix.
    Hence, the four linear equations have no nonzero common solution, so $\Lambda_1 = \emptyset$.
    It follows that $Z_1 = \bsq^{-1}(\Lambda_1) = \emptyset$. This proves (1).

    We now assume that $n \geq 2$.
    Since the four linear forms are independent, $\Lambda_n \simeq \bP^{n-2}_{\bC}$.
    The coordinate-wise squaring morphism $\bsq$ is finite and surjective.
    Therefore, its restriction $\bsq|_{Z_n} \colon Z_n \ra \Lambda_n$ is also finite and surjective.
    In particular, $Z_n$ is nonempty and $\dim Z_n = \dim \Lambda_n = n-2$.

    We next prove that $Z_n$ is smooth.
    The Jacobian matrix of $F_0, F_1, F_2, F_3$ at $x \in Z_n$ is
    \begin{equation*}
        J(x) = 2\begin{pmatrix}
          x_0 & x_1 & \cdots & x_{n+2}\\
          \lambda_0 x_0& \lambda_1 x_1& \cdots & \lambda_{n+2} x_{n+2}\\
          \lambda_0^2 x_0& \lambda_1^2 x_1& \cdots & \lambda_{n+2}^2 x_{n+2}\\
          \lambda_0^3 x_0& \lambda_1^3 x_1& \cdots & \lambda_{n+2}^3 x_{n+2}
        \end{pmatrix}.
    \end{equation*}
    Assume that $J(x)$ has rank smaller than $4$ at some point $x \in Z_n$.
    Then there exist scalars $c_0, c_1, c_2, c_3$ not all zero, such that
    \begin{equation*}
        (c_0 + c_1 \lambda_j + c_2 \lambda_j^2 + c_3 \lambda_j^3) x_j = 0
    \end{equation*}
    for every $j$.
    Set
    \begin{equation*}
        h(t) := c_0 + c_1t + c_2 t^2 + c_3 t^3.
    \end{equation*}
    This is a nonzero polynomial of degree at most $3$.
    If $x_j \neq 0$, then $h(\lambda_j) = 0$.
    Since the $\lambda_j$ are pairwise distinct, at most three coordinates of $x$ can therefore be nonzero.
    Let $\{j_1, \dots, j_s \}$, $1 \leq s \leq 3$, be indices such that $x_{j_k} \ne 0$.
    The first $s$ equations defining $Z_n$ give 
    \[
        \begin{pmatrix}
         1 & \cdots & 1 \\
         \lambda_{j_1} & \cdots & \lambda_{j_s} \\
          \vdots & \ddots & \vdots \\
         \lambda_{j_1}^{\,s-1} & \cdots & \lambda_{j_s}^{\,s-1}
        \end{pmatrix}
        \begin{pmatrix}
          x_{j_1}^{\,2} \\
          x_{j_2}^{\,2} \\
          \vdots \\
          x_{j_s}^{\,2}
        \end{pmatrix}
         =
        \begin{pmatrix}
          0 \\
          0 \\
          \vdots \\
          0
        \end{pmatrix}.
       \]
    The matrix on the left is an invertible Vandermonde matrix. 
    Hence $x_{j_1}^2 = \cdots = x_{j_s}^2 = 0$, contradicting the fact that $x$ is a point of the projective space.
    Thus, the Jacobian $J(x)$ has rank $4$ at every point of $Z_n$.
    By the Jacobian criterion \cite[Exercise I.5.8]{Ha}, $Z_n$ is smooth.
    Moreover, $Z_n$ is scheme-theoretically defined by four quadrics and has codimension $4$.
    Consequently, the four quadrics form a regular sequence, and $Z_n$ is a complete intersection.

    Now, let $n = 2$.
    Then $\dim Z_2 = 0$, so $Z_2$ is a smooth zero-dimensional projective scheme.
    In particular, it is finite and reduced.
    Since a scheme-theoretic fiber of $\bsq$ over a point of its branch locus is nonreduced, the reducedness of $Z_2 = \bsq^{-1}(\Lambda_2)$ implies that the point $\Lambda_2 \simeq \bP^0_{\bC}$ lies outside the branch locus.
    Therefore, $Z_2$ consists of exactly $2^{2+2} = 16$ points.
    This proves (2).

    Finally, suppose that $n \geq 3$.
    Then $\dim Z_n = n-2 \geq 1$, so $Z_n$ is a positive-dimensional smooth complete intersection.
    By \cite[Theorem 3.4]{Har62}, $Z_n$ is connected, and so irreducible.
    This proves (3).
\end{proof}

For $v = (v_0, v_1, \dots, v_{n+2}) \in \bC^{n+3}$, set
\begin{equation*}
    q_v(z) := \sum_{j = 0}^{n+2} v_j \prod_{\ell \neq j} (z - \lambda_{\ell}), \ \ c_j := \prod_{\ell \neq j} (\lambda_j - \lambda_{\ell}) \neq 0
\end{equation*}
so that $q_v(\lambda_j) = v_jc_j$.
Let $M_r(v) := \sum_j \lambda_j^r v_j$ be the {\it{$r$-th moment}}, let $e_t(\lambda) = e_t(\lambda_0, \dots, \lambda_{n+2})$ denote the $t$-th elementary symmetric polynomial in $\lambda_0, \dots, \lambda_{n+2}$, and let $\bC[z]_{\leq n+2}$ denote the vector space of the polynomials of degree at most $n+2$.

\begin{Lem}\label{interpolation_lemma}
\begin{enumerate}
    \item[(1)] The map $v \mapsto q_v$ is a linear isomorphism from $\bC^{n+3}$ to $\bC[z]_{\leq n+2}$.
    \item[(2)] Let $m_j$ denote the coefficient of $z^j$ in $q_v(z)$.
    Then, for $0 \leq t \leq 3$,
    \begin{equation*}
        m_{n+2-t} = \sum_{s = 0}^t (-1)^{t+s}e_{t-s}(\lambda)M_s(v).
    \end{equation*}
    In particular, $\deg q_v \leq n-2$ if and only if 
    \begin{equation*}
        M_0(v) = M_1(v) = M_2(v) = M_3(v) = 0.
    \end{equation*}
\end{enumerate}
\end{Lem}
\begin{proof}
    (1) Since $v$ can be recovered from $q_v$ by $v_j = q_v(\lambda_j)/c_j$, the map $v \mapsto q_v$ is injective.
    Both its source and target have dimension $n + 3$, so it is an isomorphism between two vector spaces.

    (2) 
    Let $e_t(\lambda \setminus \lambda_j) := e_t(\lambda_0, \dots, \lambda_{j-1}, \lambda_{j+1}, \dots, \lambda_{n+2})$ denote the $t$-th elementary symmetric polynomial in all $\lambda_{\ell}$ except $\lambda_j$.
    For each $j$, by the standard generating function for elementary symmetric polynomials \cite[Chapter I, Section 2, equation (2.2)]{Mac95},
    \begin{equation*}
        \prod_{\ell \neq j}(z - \lambda_{\ell}) = \sum_{t=0}^{n+2}(-1)^t e_t(\lambda \setminus \lambda_j)z^{n+2-t}.
    \end{equation*}
    The standard generating function for elementary symmetric polynomials also gives us
    \begin{equation*}
        e_t(\lambda \setminus \lambda_j) = \sum_{s = 0}^t (- \lambda_j)^s e_{t-s}(\lambda).
    \end{equation*}
    Therefore,
    \begin{equation*}
        \begin{split}
            m_{n+2-t} &= (-1)^t \sum_{j=0}^{n+2}v_je_t(\lambda \setminus\lambda_j)\\
                      &= (-1)^t \sum_{j=0}^{n+2}v_j\sum_{s = 0}^t (- \lambda_j)^s e_{t-s}(\lambda)\\
                      &= \sum_{s=0}^t (-1)^{t+s}e_{t-s}(\lambda)\sum_{j=0}^{n+2}\lambda_j^s v_j\\
                      &=\sum_{s=0}^t (-1)^{t+s}e_{t-s}(\lambda)M_s(v).
        \end{split}
    \end{equation*}
    In particular, the four leading coefficients $m_{n+2}, m_{n+1}, m_n, m_{n-1}$ are related to $M_0(v), M_1(v), M_2(v), M_3(v)$.
    Consequently, $\deg q_v \leq n-2$ if and only if $ M_0(v) = M_1(v) = M_2(v) = M_3(v) = 0$.
\end{proof}

Recall that, for a projective point $x = [x_0 : x_1 : \cdots : x_{n+2}] \in \bP_{\bC}^{n+2}$, we have the polynomial $p_x(z)$ (as \zcref{poly_p}) defined as
\begin{equation*}
    p_x(z) := \sum_{j = 0}^{n+2} x_j^2 \prod_{l \neq j} (z - \lambda_l).
\end{equation*}
Applying Lemma \zcref{interpolation_lemma} to $v = (x_0^2, x_1^2, \dots, x_{n+2}^2) \in \bC^{n+3}$, we obtain
\begin{equation}\label{Z_degree_equiv}
    x \in Z_n \ \ \ \text{if and only if}\ \ \ \deg p_x \leq n-2.
\end{equation}
Moreover, if $x_j \neq 0$, then $p_x(\lambda_j) = x_j^2c_j \neq 0$. Therefore, $p_x \neq 0$.
We denote by $P_x^{[d]}(Z, T)$ with $[Z:T] \in \bP^1_{\bC}$ the homogenization of $p_x(z)$ with $\deg(p_x(z)) \leq d$  in general. 
\begin{Prop}\label{W_irreducible_prop}
    Assume $n \geq 2$.
    Define
     \begin{equation*}
         \nu_n \colon \bP^1_{\bC} \times Z_n \ra \bP^{n+2}_{\bC}
     \end{equation*}
    by
    \begin{equation*}
        \nu_n([a:b], [x_0: x_1 : \cdots : x_{n+2}]) = [(b\lambda_0 -a)x_0 : \cdots : (b \lambda_{n+2} - a)x_{n+2}].
    \end{equation*}
    Then, the image of $\nu_n$ is $W$.
    Consequently, if $n \geq 3$, then $W$ is irreducible.
\end{Prop}
\begin{proof}
    We first check that $\nu_n$ is well-defined.
    Since $\lambda_j$ are pairwise distinct, for a fixed point $[a:b] \in \bP^1_{\bC}$, at most one of the numbers $b\lambda_j - a$ can vanish.
    Suppose that all the coordinates $(b\lambda_j-a)x_j$ were zero.
    Then, $x_j = 0$ for all but possibly one index.
    But $x \in Z_n$ satisfies $\sum x_j^2 = 0$.
    The point $x \in Z_n$ with at most one nonzero coordinate cannot satisfy this equation unless $x_j = 0$ for all $j$.
    This contradicts the fact that $x$ lies in projective space.
    Therefore, $\nu_n$ is well-defined.

    Next, we will show that the image of $\nu_n$ is contained in $W$.
    Let 
    \begin{equation*}
        w = \nu_n([a:b], x), \ \ w_j = (b\lambda_j-a)x_j.
    \end{equation*}
    Since $x \in Z_n$, one has $\sum \lambda_j^r x_j^2 = 0$ for $r = 0, 1, 2, 3$.
    It follows that
    \begin{equation*}
        \begin{split}
            \sum_jw_j^2 &= \sum_j (b\lambda_j - a)^2x_j^2\\
                        &= b^2\sum_j \lambda_j^2 x_j^2 -2ab \sum_j \lambda_j x_j^2 + a^2 \sum_j x_j^2\\
                        &= 0.
        \end{split}
    \end{equation*}
    Similarly,
     \begin{equation*}
        \begin{split}
            \sum_j \lambda_jw_j^2 &= \sum_j  \lambda_j (b\lambda_j - a)^2x_j^2\\
                        &= b^2\sum_j \lambda_j^3 x_j^2 -2ab \sum_j \lambda_j^2 x_j^2 + a^2 \sum_j \lambda_j x_j^2\\
                        &= 0.
        \end{split}
    \end{equation*}
    Thus $w \in X$.
    In particular, $\deg p_w(z) \leq n$.
    Since $x \in Z_n$, by \zcref{Z_degree_equiv}, Lemma \zcref{interpolation_lemma} tells us $\deg p_x \leq n-2$.
    For every $j$, we have
    \begin{equation*}
        p_w(\lambda_j) = w_j^2c_j = (b\lambda_j-a)^2 x_j^2c_j = (b\lambda_j-a)^2p_x(\lambda_j).
    \end{equation*}
    Hence, two polynomials $p_w(z)$ and $(bz-a)^2p_x(z)$ agree at the $n+3$ distinct points $\lambda_0, \lambda_1, \dots, \lambda_{n+2}$.
    Both polynomials have degree at most $n$, so they must be equal. That is
    \begin{equation*}
        p_w(z) = (bz-a)^2p_x(z).
    \end{equation*}
    After homogenizing, we obtain
    \begin{equation*}
        P_w^{[n]}(Z, T) = (bZ -a T)^2P_x^{[n-2]}(Z, T).
    \end{equation*}
    Thus, $P_w^{[n]}$ has the multiple root $[a:b]$.
    Therefore, $f(w)$ lies in the discriminant hypersurface $\Delta$, and hence $w \in f^{-1}(\Delta)_{red} = W$.
    This proves $\nu_n(\bP^1_{\bC} \times Z_n) \subseteq W$.

    Finally, we will check that every point of $W$ lies in the image.
    Let $w \in W$.
    Then, the non-zero degree-$n$ homogeneous polynomial $P_w$ has a multiple root.
    Choose one such root $[a:b] \in \bP^1_{\bC}$.
    There is a degree-$(n-2)$ homogeneous polynomial $R(Z, T)$ such that
    \begin{equation*}
        P_w^{[n]}(Z, T) = (bZ-aT)^2R(Z,T).
    \end{equation*}
    For each $j$, define $x_j^2$ by
    \begin{equation*}
        x_j^2 := \frac{R(\lambda_j, 1)}{c_j}.
    \end{equation*}
    Since the base field is $\bC$, we may choose a square root $x_j$ of each of these numbers.
    The vector $(x_0, x_1, \dotsm x_{n+2}) \in \bC^{n+3}$ is nonzero.
    Indeed, if all $x_j$ were zero, then $R(\lambda_j, 1) = 0$ for all $j$.
    But the polynomial $R(z, 1)$ has degree at most $n-2$, and would vanish at the $ n+3$ distinct points $\lambda_j$, forcing $R = 0$, a contradiction.
    By construction, 
    \begin{equation*}
        p_x(\lambda_j) = x_j^2c_j = R(\lambda_j, 1).
    \end{equation*}
    The polynomial $p_x$ has degree at most $n+2$, while $R(z, 1)$ has degree at most $n-2$.
    Since they agree at $n+3$ distinct points, they are equal. That is
    \begin{equation*}
        p_x(z) = R(z, 1).
    \end{equation*}
    Therefore, $\deg p_x \leq n-2$.
    By \zcref{Z_degree_equiv}, Lemma \zcref{interpolation_lemma} implies $x \in Z_n$.
    Finally, evaluating the factorization of $P_w^{[n]}$ at $[\lambda_j:1]$, we obtain
    \begin{equation*}
        w_j^2c_j = P_w^{[n]}(\lambda_j, 1) = (b\lambda_j-a)^2R(\lambda_j, 1) = (b\lambda_j-a)^2x_j^2c_j.
    \end{equation*}
    Since $c_j \neq 0$, $w_j^2 = (b\lambda_j-a)^2x_j^2$.
    For every $j$ with $b\lambda_j-a \neq0$, we may choose the sign of $x_j$ so that
    \begin{equation*}
        w_j = (b\lambda_j-a)x_j.
    \end{equation*}
    If $b\lambda_j -a = 0$, then $w_j = 0$, and the same formula holds automatically.
    Choosing the signs of $x_j$ does not affect membership of $Z_n$, since its defining equations involve only the squares $x_j^2$.
    Thus
    \begin{equation*}
        w = \nu_n([a:b], x).
    \end{equation*}
    Every closed point of $W$ therefore lies in the image of $\nu_n$.
    Since $\nu_n$ is projective, its image is closed, and hence $\nu_n(\bP^1 \times Z_n) = W$.

    Assume $n \geq 3$.
    By Lemma \zcref{Z_irreducibility_lemma} (3), $Z_n$ and $\bP^1_{\bC} \times Z_n$ are irreducible.
    Therefore, $W = \nu_n(\bP^1_{\bC} \times Z_n)$ is also irreducible.
\end{proof}

\begin{Prop}\label{X_and_W}
    \begin{enumerate}
        \item[(1)] If $n = 1$, then $X$ is an elliptic curve in $\bP^3_{\bC}$, and $W$ is empty.
        \item[(2)] If $n = 2$, then $X$ is a degree $4$ del Pezzo surface $dP_5$ in $\bP^4_{\bC}$, and $W$ is the union of $16$ $(-1)$-curves.
        \item[(3)] If $n \geq 3$, then $W$ is an irreducible divisor on $X$, and $[W] = 4(n-1)[H]$ with $H = c_1(\cO_X(1))$.
    \end{enumerate}
\end{Prop}
\begin{proof}
    The descriptions of $X$ in (1) and (2) are classical results (for (2), see \cite[p.550]{GH78}).

    Suppose $n = 1$.
    In this case, $\Sym^1(\bP^1_{\bC}) = \bP^1_{\bC}$ has empty discriminant locus.
    Therefore, $W = f^{-1}(\Delta)_{red} = \emptyset$.
    This proves (1).

    Next, assume $n = 2$.
    By Lemma \zcref{Z_irreducibility_lemma} (2), $Z_2$ consists of $16$ distinct reduced points.
    Therefore, we can write $Z_2 = \{z^{(1)}, \dots, z^{(16)}\}$ with the condition $z^{(i)} \neq z^{(j)}$ for $i \neq j$.
    Since $z^{(i)}\in Z_2$, \zcref{Z_degree_equiv} tells us $\deg p_{z^{(i)}} = 0$.
    Moreover, $p_{z^{(i)}} \neq 0$.
    Indeed, for $z^{(i)} := [z_0^{(i)}:z_1^{(i)}:z_2^{(i)}:z_3^{(i)}:z_4^{(i)}]$, set $v_i := ((z_0^{(i)})^2, (z_1^{(i)})^2, (z_2^{(i)})^2, (z_3^{(i)})^2, (z_4^{(i)})^2) \neq 0$.
    Then, $p_{z^{(i)}} = q_{v_i} \neq 0$ because $v \mapsto q_v$ is injective by Lemma \zcref{interpolation_lemma} (1).
    Hence, $p_{z^{(i)}}$ is nonzero constant.
    Evaluating at $\lambda_j$, we obtain $p_{z^{(i)}}(\lambda_j) = (z_j^{(i)})^2c_j \neq 0$.
    Therefore, $z_j^{(i)} \neq 0$ for $j = 0, \cdots, 4$.
    The five coordinates of $\nu_i := \nu_2|_{\bP^1_{\bC} \times \{z^{(i)} \}}$ are sections of $\cO_{\bP^1_{\bC}}(1)$.
    Any two of them, say $(b\lambda_k -a)z_k^{(i)}$ and $(b\lambda_{\ell} -a)z_{\ell}^{(i)}$, are linearly independent when $k \neq \ell$.
    Hence, these sections span $H^0(\bP^1_{\bC}, \cO_{\bP^1_{\bC}}(1))$.
    It follows that $\nu_i$ is a closed immersion, and its image $L_i := \nu_i(\bP^1_{\bC} \times \{z^{(i)} \})$ is a line in $X$.

    We need to show that different points of $Z_2$ give different lines.
    For $m = 0, 1, 2, 3, 4$, let 
    \begin{equation*}
        H_m := \{x_m = 0 \} \subset \bP^4_{\bC}
    \end{equation*}
    be the $m$-th coordinate hyperplane.
    Since $z_j^{(i)} \neq 0$, the line $L_i$ is not contained in any $H_m$, and its unique intersection with $H_m$ is 
    \begin{equation*}
        L_i \cap H_m = \nu_i([\lambda_m:1]) = [(\lambda_0-\lambda_m)z_0^{(i)}: \cdots:0:\cdots:(\lambda_4-\lambda_m)z_4^{(i)}].
    \end{equation*}
    Suppose $L_k = L_{\ell}$.
    Then, their intersections with $H_m$ are equal.
    Hence, for each $m = 0, \dots, 4$, there exists $c_m\in \bC^{\times}$ such that
    \begin{equation*}
        z_j^{(k)} = c_m z_j^{(\ell)}
    \end{equation*}
    for every $j \neq m$.
    Taking $m = 0, 1$, and comparing coordinate $j = 2$ gives $c_0 = c_1 =:c$.
    The relation for $m = 0$ gives $z_j^{(k)} = c z_j^{(\ell)}$ for $j = 1, 2, 3, 4$, while the relation for $m = 1$ also gives $z_0^{(k)} = c z_0^{(\ell)}$.
    Thus, $L_k = L_{\ell}$ implies $z^{(k)} = z^{(\ell)}$ in $Z_2$.
    Consequently, the $16$ points of $Z_2$ determine $16$ distinct lines.
    Since $\nu_2$ is surjective,
    \begin{equation*}
        W = \nu_2(\bP^1_{\bC} \times Z_2) = \bigcup_{i = 1}^{16}L_i.
    \end{equation*}
    Finally, since $X \subset \bP^4_{\bC}$ is the anti-canonical embedding, $-K_X = \cO_X(1)$.
    For each line $L_i \simeq \bP^1_{\bC}$, $(-K_X).L_i = 1$.
    Therefore, the adjunction formula gives $(L_i)^2 = -1$.
    Hence, the $16$ lines are precisely $16$ distinct $(-1)$-curves on $X = dP_5$.
    This proves (2).
    
    (3) Suppose $n \geq 3$.
    By Proposition \zcref{W_irreducible_prop}, $W$ is an irreducible divisor in $X$.
    First, let us prove that $W = f^{-1}(\Delta)_{red} = f^{-1}(\Delta) := \Delta \times_{\bP^n_{\bC}}X$ scheme theoretically.
    Since $\Delta$ is a hypersurface in $\bP^n_{\bC}$, and $f(X) \not\subset \Delta$, the scheme $f^{-1}(\Delta)$ is an effective Cartier divisor on the smooth variety $X$, with support $W$.
    The finite morphism $f|_W$ maps $W$ to $\Delta$, since this is the reduction of the morphism $f^{-1}(\Delta) \ra \Delta$ that is the base change of the finite surjective morphism $f\colon X \ra \bP^n_{\bC}$.
    The generic point of $\Delta$ avoids all branch loci $B_j$ and $f$ is \'{e}tale over $\bP^n_{\bC} \setminus \cup_jB_j$.
    Hence, $f$ is \'{e}tale at the generic point of $W$.
    It follows that $f^{-1}(\Delta)$ is reduced at its unique generic point, and therefore generically reduced.
    Since an effective Cartier divisor on a smooth variety is Cohen--Macaulay, $f^{-1}(\Delta)$ satisfies Serre's $(S_1)$ condition.
    The conditions $(R_0)+(S_1)$ imply that $f^{-1}(\Delta)$ is reduced  \cite[\href{https://stacks.math.columbia.edu/tag/031R}{Tag 031R}]{stacks-project}.
    Thus, $W = f^{-1}(\Delta) := \Delta \times_{\bP^n_{\bC}}X$ scheme theoretically.
    
    Since $\deg \Delta = 2(n-1)$ in $\bP^n_{\bC}$ and $f^*\cO_{\bP^n_{\bC}}(1) \simeq \cO_X(2)$ with $H = c_1(\cO_X(1))$, we have $[W] = 4(n-1)[H]$.
    This proves (3).
\end{proof}

\begin{Rem}
{\rm
 \begin{enumerate}
     \item Proposition \zcref{X_and_W} (3) agrees with \cite[Section 5.1.2]{PP21} when $n = 3$.
     \item Donagi--Pantev identified $W$ with the union of $16$ $(-1)$-curves when $n = 2$ \cite[Theorem C]{DP}.
 \end{enumerate}
}
\end{Rem}

\begin{Thm}\label{first_homology_group}
\begin{enumerate}
    \item[(1)] If $n = 1$, then $H_1 (X \setminus W, \bZ) = H_1(X,\bZ) \simeq \bZ^2$.
    \item[(2)] If $n = 2$, then $H_1(X \setminus W, \bZ) \simeq \bZ^{10}$.
    \item[(3)] If $n \geq 3$, then $H_1(X \setminus W, \bZ) \simeq \bZ/4(n-1)\bZ$.
\end{enumerate}
\end{Thm}
\begin{proof}
    (1) Suppose $n = 1$. 
    Then, the statement follows from Proposition \zcref{X_and_W} (1). 

    (2) Assume $n = 2$.
    By Proposition \zcref{X_and_W} (2), we know that $X$ is a degree $4$ del Pezzo surface $dP_5$ in $\bP^4_{\bC}$, and $W$ is the union of $16$ $(-1)$-curves.
    Blowing down the five exceptional curves identifies
    \begin{equation*}
        X \setminus W \simeq \bP^2_{\bC} \setminus \left( \bigcup_{1\leq i < j \leq 5}L_{ij} \cup Q \right),
    \end{equation*}
    where the $L_{ij}$ are ten lines and $Q$ is the conic through the five blowing-down points.
    By \cite[Section 4, Proposition (1.3)]{Dimca92}, for a reduced plane curve $C$ with irreducible components $C_1, \dots, C_r$ of degree $d_1, \dots, d_r$, one has
    \begin{equation*}
        H_1(\bP^2_{\bC} \setminus C, \bZ) \simeq \bZ^{r-1} \oplus \bZ/\gcd(d_1, \dots, d_r)\bZ.
    \end{equation*}
    In the present case, $r = 11$ and the component degrees are $1, \cdots, 1, 2$.
    Hence, $H_1(X \setminus W, \bZ) \simeq \bZ^{10}$.

    (3) Since $X \subset \bP^{n+2}_{\bC}$ is a smooth complete intersection of dimension $n \geq 3$, the weak Lefschetz theorem gives
    \begin{equation*}
        H_k(X, \bZ) \simeq H_k(\bP^{n+2}_{\bC}, \bZ)
    \end{equation*}
    for $k < n$.
    In particular, 
    \begin{equation*}
        H_1(X, \bZ) = 0, \ \ H_2(X, \bZ) \simeq \bZ.
    \end{equation*}
    Moreover, 
    \begin{equation*}
        H^2(\bP^{n+2}_{\bC}, \bZ) \xrightarrow{\sim} H^2(X, \bZ)  ,
    \end{equation*}
    so the hyperplane class $H = c_1(\cO_X(1))$ is a primitive generator.
    The Grothendieck--Lefschetz theorem for Picard groups yields $\Pic(X) = \bZ[\cO_X(1)]$.

    Let us compute $H_1(X \setminus W, \bZ)$.
    $W$ may have singular locus $\Sing(W)$.
    Set $X^{\circ} := X \setminus \Sing(W)$, $W_{reg} := W \setminus \Sing(W)$.
    Then $W_{reg}$ is a closed connected smooth divisor in $X^{\circ}$, and $X^{\circ} \setminus W_{reg} = X \setminus W$.
    Since $W$ is reduced, $\Sing(W)$ has complex codimension at least two in $X$, and hence $H_j(X, X^{\circ}, \bZ) = 0$ for $j \leq 3$.
    The exact sequence of the triple $X \setminus W \subset X^{\circ} \subset X$ now identifies
    \begin{equation*}
        H_2(X, X\setminus W, \bZ) \simeq H_2(X^{\circ}, X\setminus W, \bZ) \simeq H_0(W_{reg}, \bZ) \simeq \bZ.
    \end{equation*}
    Since $W$ is an irreducible divisor of class $4(n-1)[H]$ in $X$ by Proposition \zcref{X_and_W} (3), the homology long exact sequence for the pair $(X, X \setminus W)$ gives
    \begin{equation*}
        H_2(X, \bZ) \ra H_2(X, X \setminus W, \bZ) \simeq \bZ \ra H_1(X \setminus W, \bZ) \ra 0.
    \end{equation*}
    The first map is intersection with $[W] = 4(n-1)H$, hence multiplication by $4(n-1)$ after primitive generators are chosen.
    The cokernel is the stated cyclic group.
\end{proof}

\section{Fundamental Group of $X \setminus W$}\label{fundamental_group_section}
In this section, by using a {\it{root stack and quotient stack description}} (Theorem \zcref{root=quotient}), we identify the fundamental group $\pi_1(X \setminus W)$ of $X\setminus W$ with a subgroup of an orbifold mixed braid group (Corollary \zcref{ker_description_of_pi_1}).
\subsection{Branch locus and Ramification locus}
Recall that $U := \bP^n_{\bC} \setminus \Delta$, $Y := X \setminus W$, and for each $j = 0, 1,  \dots, n+2$, we have
\begin{equation*}
    B_j := \{D \in |\cO_{\bP^1}(n)|  \ | \ \lambda_j \in \Supp(D)\} \subset \bP^n_{\bC}, \ \ R_j := X \cap \{x_j = 0 \}.
\end{equation*}
By definition of $f \colon X \ra \bP^n_{\bC}$, the union $B := \cup_{j = 0}^{n+2} B_j$ is the branch locus of $f$, and the union $R := \cup_{j= 0}^{n+2} R_j$ is the ramification locus of $f$. 
Let us denote $f_Y \colon Y \ra U$ the restriction of $f$ to $Y = X \setminus W$, and denote $B_j^{\circ} := B_j \cap U$, $B^{\circ} := B\cap U = \cup_{j=0}^{n+2}B_j^{\circ}$, $R_j^{\circ} := R_j \cap Y$, and $R^{\circ} := R \cap Y = \cup_{j = 0}^{n+2} R_j^{\circ}$. 

By the equations in \zcref{polynomial_p_x_lambda}, 
\begin{equation}\label{divisor_root_relation}
    f_Y^* B_j^{\circ} = 2 R_j^{\circ}.
\end{equation}

\begin{Prop}\label{transversal_prop}
    Let $J := \{j_1, \dots, j_k \}$.
    If $k \leq n$, then
    \begin{equation*}
        B_J := B_{j_1} \cap \cdots \cap B_{j_k}
    \end{equation*}
    is a projective linear subspace of codimension $k$, and the hyperplanes meet transversally.
    If $k \geq n +1$, the intersection is empty.
    Hence, $B^{\circ} = \cup_{j = 0}^{n+2} B^{\circ}_j$ is a simple normal-crossing divisor on $U = \bP^n_{\bC} \setminus \Delta$.
\end{Prop}
\begin{proof}
    Each $B_j \subset \bP(\bC[z]_{\leq n}) = \bP^n_{\bC}$ is defined by the evaluation form of polynomials 
    \begin{equation*}
        l_j \colon \bC[z]_{\leq n} \ra \bC, \ \  p(z) \mapsto p(\lambda_j). 
    \end{equation*}
    Since $\lambda_j$ are distinct, any $k \leq n + 1$ such forms are linearly independent by the Vandermonde determinant.
    Thus, $B_J \simeq \bP^{n-k}_{\bC}$ for $k \leq n$, whereas $B_J = \emptyset$ for $k \geq n+1$.
    At each point, the local defining equations of the hyperplanes through that point form a part of a coordinate system.
    Hence, the intersections are transverse, and $B^{\circ} = B \cap U$ is a simple normal-crossing divisor on $U$.
\end{proof}

If $x  \in Y$, define
\begin{equation*}
    J(x) := \{j \ | \ x_j = 0 \}, \ \ k := |J(x)|.
\end{equation*}
The integer $k$ $(0 \leq k \leq n )$ is the number of branch components that pass through $f_Y(x) \in U$, or equivalently, the codimension of the corresponding stratum.

\begin{Lem}\label{local_square_chart}
    Let $x \in Y$, $u = f_Y(x) \in U$, and $J(x) = \{j_1, \dots, j_k \}$.
    After passing to \'{e}tale neighborhoods, there are coordinates
    \begin{equation*}
        (\zeta_1, \dots, \zeta_k, r_1, \dots, r_{n-k}) \ \ \text{at} \ x
    \end{equation*}
    and
    \begin{equation*}
        (\tau_1, \dots, \tau_k, r'_1, \dots, r'_{n-k}) \ \ \text{at} \ u
    \end{equation*}
    such that
    \begin{equation*}
        B_{j_\ell}^{\circ} = \{ \tau_{\ell} = 0\}, \ \ R_{j_{\ell}}^{\circ} = \{\zeta_{\ell}  = 0\},
    \end{equation*}
    and
    \begin{equation*}
        \tau_{\ell} = \zeta_{\ell}^2 \ (1 \leq \ell \leq k), \ \ r_m \mapsto r'_m \ (1 \leq m \leq n-k).
    \end{equation*}
\end{Lem}
\begin{proof}
    Put $J = J(x) = \{j_1, \dots, j_k \}$, and set
    \begin{equation*}
        R_J := \bigcap_{\ell=1}^k R_{j_\ell}, \ \  R_J^{\circ} := \bigcap_{\ell=1}^k R_{j_\ell}^{\circ}, \ \ B_J:= \bigcap_{\ell = 1}^k B_{j_{\ell}},\ \ B_J^{\circ}:= \bigcap_{\ell = 1}^k B_{j_{\ell}}^{\circ},
    \end{equation*}
    \begin{equation*}
        R_J^{str} := R_J^{\circ} \setminus \left( \bigcup_{i \notin J}R_i^{\circ} \right), \ \ B_J^{str} := B_J^{\circ} \setminus \left( \bigcup_{i \notin J}B_i^{\circ} \right)
    \end{equation*}
    Since $J = J(x)$, no component $R_i^{\circ}$ with $i \notin J$ passes through $x$, and no component $B_i^{\circ}$ with $i \notin J$ passes through $u = f_Y(x)$.
    Therefore, $x \in R_J^{str}$ and $u \in B_J^{str}$.
    After shrinking around $x$ and $u$, we may assume that the only ramification and branch components meeting the respective neighborhood are those indexed by $J$.

    By Proposition \zcref{transversal_prop}, the divisors $B_{j_1}^{\circ}, \dots, B_{j_k}^{\circ}$ meet transversely.  
    Hence, $B_J^{\circ}$ itself is a smooth closed subvariety of $U$ of codimension $k$.
    The stratum $B_J^{str}$ is an open subvariety of $B_J^{\circ}$, and is also smooth of dimension $n-k$ near $u$.
    Consequently, we may choose coordinates $ (\tau_1, \dots, \tau_k, r'_1, \dots, r'_{n-k})$ at $u$ such that $B_{j_{\ell}}^{\circ} = \{\tau_{\ell} = 0 \}$ for $1 \leq \ell \leq k$ and such that the restrictions of $r'_1, \dots, r'_{n-k}$ form \'{e}tale coordinates along $B_J^{str}$.

    We next consider the corresponding coordinates on $Y$.
    Choose an index $i_0 \not\in J$.
    Since $J = J(x)$, $x_{i_0} \neq 0$.
    On the affine chart $\{x_{i_0} \neq 0 \}$ of $\bP^{n+2}_{\bC}$, put
    \begin{equation*}
        z_{\ell} = \frac{x_{j_{\ell}}}{x_{i_0}}
    \end{equation*}   
    for $1 \leq \ell \leq k$.
    Thus, $R_{j_{\ell}}^{\circ}$ is locally defined by $z_{\ell} = 0$.
    The variety $R_J^{\circ}$ is smooth of codimension $k$ in Y.
    Indeed, since $x_j = 0$ for $j \in J$, $R_J^{\circ}$ is given by two equations in the remaining coordinates in $Y$, that is,
    \begin{equation*}
        R_J^{\circ} = \left\{ \sum_{i \notin J} x_i^2 = 0, \ \sum_{i \notin J}\lambda_i x_i^2 = 0 \right\}.
    \end{equation*}    
    If the Jacobian had rank smaller than $2$ at a point, there would be scalars $a, b$ not both zero, such that
    \begin{equation*}
        (a + b\lambda_i)x_i = 0
    \end{equation*}
    for every $i \notin J$.
    Since the $\lambda_i$ are pairwise distinct, at most one of the remaining coordinates $x_i$ could be nonzero.
    However, with the condition $\sum_{i \notin J} x_i^2 =0$, it is impossible.
    Thus, $R_J^{\circ}$ is smooth of dimension $n-k$, and the functions $z_1, \dots, z_k$ are a part of local coordinates along $R_J^{\circ}.$
    We now show that the restriction 
    \begin{equation*}
        f_J^{\circ} := f|_{R_J^{\circ}} \colon R_J^{\circ} \ra B_J^{\circ}
    \end{equation*}
    is the coordinatewise-squaring morphism in the remaining coordinates $\{x_i \}_{i \notin J}$.
    Under the natural identification
    \begin{equation*}
        R_J \simeq \left\{\sum_{i \notin J}x_i^2 = 0, \ \sum_{i \notin J} \lambda_i x_i^2 = 0  \right\},\ \  B_J \simeq \left\{\sum_{i \notin J}y_i = 0, \ \sum_{i \notin J} \lambda_i y_i = 0  \right\},
    \end{equation*}
    the restriction $f_J := f|_{R_J} \colon R_J \ra B_J$ is the coordinatewise-squaring map $[x_i]_{i \notin J} \mapsto [x_i^2]_{i\notin J}$.
    Since $Y = f^{-1}(U)$, $R_J^{\circ} = R_J \cap Y$, and $B_J^{\circ} = B_J \cap U$, $f_J^{\circ} \colon R_J^{\circ} \ra B_J^{\circ}$ is the base change of $f_J$, and again the coordinatewise-squaring morphism.
    Moreover, the restriction
    \begin{equation*}
        f_J^{str} := f_J|_{R_J^{str}} \colon R_J^{str} \ra B_J^{str}
    \end{equation*}
    is finite \'{e}tale near $x$.
    After a further \'{e}tale base change, a finite \'{e}tale morphism becomes a disjoint union of copies of the base.
    Choosing the copy containing $x$, we may therefore identify the chosen \'{e}tale neighborhoods of $R_J^{str}$ and $B_J^{str}$.
    In particular, we may arrange that 
    \begin{equation*}
        f_Y^*(r'_m) = r_m
    \end{equation*}
    for $1 \leq m \leq n-k$, where $r_1, \dots, r_{n-k}$ are coordinates on the stratum $R_J^{str}$ near $x$.
    Consequently, $(z_1, \dots, z_k, r_1, \dots, r_{n-k})$ is an \'{e}tale coordinate system at $x$.

    For each $\ell$, the identity \zcref{divisor_root_relation}
    \begin{equation*}
        f_Y^* B_{j_{\ell}}^{\circ} = 2 R_{j_{\ell}}^{\circ}
    \end{equation*}
    implies that the pullback of the local equation $\tau_{\ell}$ has the form
    \begin{equation*}
        f_Y^*(\tau_{\ell}) = \epsilon_{\ell} z_{\ell}^2
    \end{equation*}
    for some unit $\epsilon_{\ell} \in \cO_{Y, x}^{\times}$.
    After another \'{e}tale base change, we may choose a square root $s_{\ell} \in \cO_{Y, x}^{\times}$ of $\epsilon_{\ell}$, that is, $s_{\ell}^2 = \epsilon_{\ell}$.
    Set $\zeta_{\ell} := s_{\ell}z_{\ell}$.
    Then,
    \begin{equation*}
        f_Y^*(\tau_{\ell}) = \zeta_{\ell}^2.
    \end{equation*}
    Multiplying a local parameter by a unit does not change its zero divisor, $R_{j_{\ell}}^{\circ} = \{\zeta_{\ell} = 0 \}$.
    
    Therefore,
     \begin{equation*}
        (\zeta_1, \dots, \zeta_k, r_1, \dots, r_{n-k}) 
    \end{equation*}
    and
    \begin{equation*}
        (\tau_1, \dots, \tau_k, r'_1, \dots, r'_{n-k}) 
    \end{equation*}
    are \'{e}tale coordinate systems at $x$ and $u$ that satisfy the claim of the lemma.
\end{proof}

\subsection{Root stacks and Quotient stacks}
Following \cite{Cad07}, define the root stack
\begin{equation*}
    \cU_n := U \left[ \sqrt[2]{(B_0^{\circ}, s_0)}, \dots, \sqrt[2]{(B_{n+2}^{\circ}, s_{n+2})} \right]
\end{equation*}
over $U = \bP^n_{\bC} \setminus \Delta$.
Here, $s_j$ is a section of $\cO_{U}(B_j^{\circ})$ defining $B_j^{\circ}$.

We also consider the quotient stack $[Y/G]$.
\begin{Lem}\label{separated_of_finite_type}
    Both $[Y/G]$ and $\cU_n$ are separated Deligne--Mumford stacks of finite type over $\bC$.
\end{Lem}
\begin{proof}
    Since $Y$ is an open subscheme of the projective variety $X$, it is separated and of finite type over $\bC$.
    The group $G$ is finite \'{e}tale.
    Hence the quotient atlas $p \colon Y \ra [Y/G]$ is representable, finite \'{e}tale and surjective, so $[Y/G]$ is Deligne--Mumford.

    Notice that we have the following cartesian diagram
    \[\begin{tikzcd}
	{G \times Y} && {Y \times Y} \\
	& \square \\
	{[Y/G]} && {[Y/G] \times [Y/G].}
	\arrow["{(\sigma, p_2)}", from=1-1, to=1-3]
	\arrow[from=1-1, to=3-1]
	\arrow["{p \times p}", from=1-3, to=3-3]
	\arrow["\Delta_{[Y/G]}", from=3-1, to=3-3]
\end{tikzcd}\]
    That is, after base change along
    \begin{equation*}
       p \times p \colon  Y \times Y \ra [Y/G] \times [Y/G],
    \end{equation*}
    the diagonal of $[Y/G]$ becomes
    \begin{equation}\label{base_changed_diagonal}
        (\sigma, p_2) \colon G \times Y \ra Y \times Y, \ (g, y) \mapsto (gy, y).
    \end{equation}
   For each $g \in G$, the restriction $\{g \} \times Y$ is the graph of the automorphism $g \colon Y \ra Y$, hence a closed immersion because $Y$ is separated.
   Since $G$ is finite, the morphism \zcref{base_changed_diagonal} is finite.
   Thus $\Delta_{[Y/G]}$ is finite, and $[Y/G]$ is separated.

   For $\cU_n$, choose an \'{e}tale neighborhood $V = \Spec(A) \ra U$ on which the root divisors through the chosen point are defined by $\tau_1, \dots, \tau_{k}$.
   The standard root-stack chart gives
   \begin{equation}\label{root_stack_chart}
       \cU_n \times_U V \simeq [V^{\sharp}/\mu_2^k], \ \ V^{\sharp} = \Spec(A[w_1, \dots, w_k]/(w_1^2-\tau_1, \dots, w_k^2-\tau_k)).
   \end{equation}
   The scheme $V^{\sharp}$ is separated and of finite type over $\bC$, and $\mu_2^k$ is finite \'{e}tale.
   The local quotient charts provide an \'{e}tale atlas, so $\cU_n$ is Deligne--Mumford.
   By construction, $U$ is the coarse moduli space of the Deligne-- Mumford stack $\cU_n$ with the structure morphism $\pi_n \colon \cU_n \ra U$.
   
   The same graph argument shows that the diagonal of this local quotient over $V$ is finite.
   Finiteness descends \'{e}tale locally, so $\Delta_{\cU_n/U}$ is finite.
   Hence $\pi_n \colon \cU_n \ra U$ is separated.
   Since $U$ is separated over $\bC$, $\cU_n$ is separated over $\bC$.
   
 \end{proof}

\begin{Prop}\label{lift_of_f_Y}
    There exists a lift $\tilde{f}_Y \colon Y \ra \cU_n$ of $f_Y \colon Y \ra U$ that makes the following diagram
   \begin{equation}
\xymatrix{
Y\ar[rr]^-{\tilde{f}_Y}_-{}\ar[dr]_-{f_Y}&\ar@{}[d]|{\circlearrowright}&\cU_n\ar[dl]^-{\pi_n}\\
&U&
}
\end{equation}
commutative.
Moreover, the morphism $\tilde{f}_Y$ is representable by algebraic spaces, finite, and \'{e}tale.
\end{Prop}
\begin{proof}
Equation \zcref{divisor_root_relation} gives us square-root data
\begin{equation}\label{square_root_data}
    M_j := \cO_Y(R_j^{\circ}), \ \ t_j \in \Gamma(Y, M_j), \ \ \alpha_j \colon M_j^{\otimes 2} \xrightarrow{\sim} f_Y^* \cO_U(B_j^{\circ}),
\end{equation}
such that $\alpha_j(t_j^{\otimes 2}) = f_Y^* s_j$.
Hence, by the universal property of the root stack, we obtain a lift $\tilde{f}_Y \colon Y \ra \cU_n$ such that $\pi_n \circ \tilde{f}_Y = f_Y$.

Let us prove the representability of the morphism $\tilde{f}_Y$.
For a morphism $T \ra \cU_n$ from a scheme $T$, consider the fiber product $Y \times_{\cU_n}T$. This fiber product fits in the following diagram
 \[\begin{tikzcd}
	{Y \times_{\mathcal{U}_n}T} && {Y \times T} \\
	& \square \\
	{\mathcal{U}_n} && {\mathcal{U}_n \times \mathcal{U}_n}
	\arrow[from=1-1, to=1-3]
	\arrow[from=1-1, to=3-1]
	\arrow[from=1-3, to=3-3]
	\arrow["{\Delta_{\mathcal{U}_n}}", from=3-1, to=3-3]
\end{tikzcd}\]
Since the diagonal of an algebraic stack is representable, $Y \times_{\cU_n}T$ is an algebraic space.
Therefore, $\tilde{f}_Y$ is representable.

Next, we will check the finiteness.
Since the graph $\Gamma_{\tilde{f}_Y} \colon Y \ra Y \times_U \cU_n$ of $\tilde{f}_Y$ is the base change of the diagonal map $\Delta_{\cU_n/U} \colon \cU_n \ra \cU_n \times_U \cU_n$, and $\Delta_{\cU_n/U}$ is finite, $\Gamma_{\tilde{f}_Y}$ is finite.
Also, the second projection $p_2 \colon Y \times_U \cU_n \ra \cU_n$ is the base change of the finite morphism $f_Y \colon Y \ra U$, so $p_2$ is finite.
Therefore, the composition $\tilde{f}_Y = p_2 \circ \Gamma_{\tilde{f}_Y}$ is finite.

It remains to prove that the morphism $\tilde{f}_Y$ is \'{e}tale.
Let $x \in Y$, $u = f_Y(x)$, and write $J(x) = \{j_1, \dots, j_k \}$.
By Lemma \zcref{local_square_chart}, we may choose \'{e}tale neighborhoods
$\iota_Y \colon \widetilde{V} \ra Y$ of $x$ and $\iota_U \colon V = \Spec(A) \ra U$ of $u$, together with a morphism $f_{\widetilde{V}} \colon \widetilde{V} \ra V$ satisfying $f_Y \circ \iota_Y = \iota_U \circ f_{\widetilde{V}}$.
We may also choose \'{e}tale coordinate maps
\begin{equation*}
    \varphi = (\zeta_1, \dots, \zeta_k, r_1, \dots, r_{n-k}) \colon \widetilde{V}
 \ra \bA_{\zeta, r}^n
\end{equation*}
and
\begin{equation*}
    \psi = (\tau_1, \dots, \tau_k, r'_1, \dots, r'_{n-k}) \colon V \ra \bA_{\tau, r'}^n
\end{equation*}
such that $f_{\widetilde{V}}^* \tau_{\ell} = \zeta_{\ell}^2$ for $1 \leq \ell \leq k$ and $f_{\widetilde{V}}^* r'_m = r_m$ for $1 \leq m \leq n-k$.
Here, we can write $B_{j_{\ell}}^{\circ} = \{\tau_{\ell} = 0\}$ and $R_{j_{\ell}}^{\circ} = \{\zeta_{\ell} = 0 \}$.
After shrinking, we may also assume that $V$ meets no branch component indexed outside $J(x)$.
Consider the standard root chart
\begin{equation*}
    V^{\sharp} := \Spec(A[w_1, \dots, w_k]/(w_1^2-\tau_1, \dots, w_k^2-\tau_k)).
\end{equation*}
The identities $f_{\widetilde{V}}^* \tau_{\ell} = \zeta_{\ell}^2$ define a morphism $h \colon \widetilde{V} \ra V^{\sharp}$ over $V$, with $h^*w_{\ell} = \zeta_{\ell}$.

We first show that $h$ is \'{e}tale.
Let us define the map
\begin{equation*}
    s \colon \bA_{\zeta,r}^n \ra \bA_{\tau,r'}^n, \ \ (\zeta_1, \dots, \zeta_k, r_1, \dots, r_{n-k}) \mapsto (\zeta_1^2, \dots, \zeta_k^2, r_1, \dots, r_{n-k}).
\end{equation*}
By construction, $V^{\sharp} \simeq V \times_{\bA_{\tau,r'}^n} \bA_{\zeta, r}^n$, where the fiber product is taken with respect to $\psi$ and $s$.
Then, the projection $\varphi^{\sharp} \colon V^{\sharp} \ra \bA_{\zeta,r}^n$ is \'{e}tale, because it is the base change of the \'{e}tale coordinate map $\psi$.
Moreover, $\varphi^{\sharp} \circ h = \varphi$.
Since $\varphi$ and $\varphi^{\sharp}$ are \'{e}tale, $h$ is also \'{e}tale.

By the local description \zcref{root_stack_chart}, we have $\cU_n \times_U V \simeq [V^{\sharp}/\mu_2^k]$.
Under compatible local trivializations, the pullback of the square-root data \zcref{square_root_data} associated with $B_{j_{\ell}}^{\circ}$ is represented by the trivial line bundle with section $\zeta_{\ell}$, whose square is $f_{\widetilde{V}}^* \tau_{\ell}$.
The remaining root factors are trivial over $V$, since their defining sections are invertible there.
Consequently, the morphism
\begin{equation*}
    \widetilde{V} \ra \cU_n \times_U V
\end{equation*}
induced by $\tilde{f}_Y$ is $2$-isomorphic, under the above identification \zcref{root_stack_chart}, to the composite
\begin{equation*}
 \widetilde{V} \xrightarrow{h} V^{\sharp} \xrightarrow{q} [V^{\sharp}/\mu_2^k],
\end{equation*}
where $q$ is the standard quotient atlas.
Since $\mu_2^k$ is finite \'{e}tale over $\bC$, the morphism $q$ is representable, finite \'{e}tale, and surjective.
Thus, $q \circ h$ is \'{e}tale.

Finally, the projection $\cU_n \times_UV \ra \cU_n$ is \'{e}tale, because it is the base change of $\iota_U$.
It follows that $\tilde{f}_Y \circ \iota_Y$ is also \'{e}tale.
As $x$ is arbitrary, these neighborhoods form an \'{e}tale covering of $Y$.
Since $\tilde{f}_Y$ is representable, and \'{e}taleness is local on the source for the \'{e}tale topology \cite[\href{https://stacks.math.columbia.edu/tag/06M2}{Tag 06M2}]{stacks-project}, $\tilde{f}_Y$ is \'{e}tale. 
\end{proof}

\begin{Lem}[{\cite[Exercise 5.4.14]{Alper}}]\label{criterion}
    If $f \colon \cX \ra \cY$ is a morphism of Deligne--Mumford stacks separated and of finite type over a noetherian algebraic space $S$ fitting in a commutative diagram
    \[\begin{tikzcd}
	{\mathcal{X}} && {\mathcal{Y}} \\
	\\
	X && Y
	\arrow["f", from=1-1, to=1-3]
	\arrow["{\pi_{\mathcal{X}}}"', from=1-1, to=3-1]
	\arrow["{\pi_{\mathcal{Y}}}", from=1-3, to=3-3]
	\arrow[from=3-1, to=3-3]
    \end{tikzcd}\]
    where $\pi_{\cX} \colon \cX \ra X$ and $\pi_{\cY} \colon \cY \ra Y$ are coarse moduli spaces and $x \in |\cX|$ is a closed point such that 
    \begin{enumerate}
        \item[(1)] $f$ is \'{e}tale at $x$ and
        \item[(2)] the induced map $G_x \ra G’_{f(x)}$ of geometric stabilizer groups is bijective,
    \end{enumerate}
    then there exists an open neighborhood $U \subseteq X$ of $\pi_{\cX}(x)$ such that $U \ra X \ra Y$ is \'{e}tale and 
    \begin{equation*}
        \pi_{\cX}^{-1}(U) = U \times_X \cX \simeq U \times_Y \cY.
    \end{equation*}
\end{Lem}
\begin{proof}
    The statement follows from \cite[Proposition 6.5, Proposition 6.7]{Ry13}.
\end{proof}

\begin{Lem}\label{ealeness_geometric_stabilizers}
    The morphism $\tilde{f}_Y \colon Y \ra \cU_n$ naturally descends to the morphism
    \begin{equation*}
        \Phi \colon [Y/G] \ra \cU_n.
    \end{equation*}
    Moreover, the morphism $\Phi$ is representable by algebraic spaces and \'{e}tale.
\end{Lem}
\begin{proof}
    First, notice that each $R_j^{\circ}$ is a $G$-Cartier divisor. Therefore, $(\cO_Y(R_j^{\circ}), t_j)$ has its $G$-linearization, and the square-root data \zcref{square_root_data} descend to $[Y/G]$.
    As a result, the lift $\tilde{f}_Y$ naturally descends to $\Phi \colon [Y/G] \ra \cU_n$.

    Next, let us prove that the morphism $\Phi$ is representable by algebraic spaces.
    By \cite[\href{https://stacks.math.columbia.edu/tag/04YY}{Tag 04YY}]{stacks-project}, it is enough to prove that for every scheme $T$ and every object $\xi \in [Y/G](T)$, the homomorphism
    \begin{equation}\label{autom_hom}
        \Aut_{[Y/G]}(\xi) \ra \Aut_{\cU_n}(\Phi(\xi))
    \end{equation}
    has trivial kernel.
    By definition of quotient stacks, $\xi \in [Y/G](T)$ represents a $G$-torsor $P \ra T$ together with a $G$-equivariant morphism $u \colon P \ra Y$.
    Let $\alpha \in \Aut_{[Y/G]}(\xi)$ and assume that $\Phi(\alpha) \in \Aut_{\mathcal{U}_n}(\Phi(\xi))$ is the identity.
    After a suitable base change $T' \ra T$, the $G$-torsor $P$ becomes trivial.
    Choose a trivialization $P_{T'} \simeq G \times T'$.
    The equivariant morphism $u_{T'} \colon G \times T' \ra Y$ is then determined by a morphism $x \colon T' \ra Y$, and the automorphism $\alpha_{T'}$ is uniquely written as
    \begin{equation*}
        \alpha_{T'} \colon G \times T' \ra G \times T', \ \ (h, t) \mapsto (h.g(t), t)
    \end{equation*}
    with a morphism $g \colon T' \ra G$.
    The condition $u \circ \alpha = u$ implies $g(t).x(t) = x(t)$, hence $g(t)$ belongs to the stabilizer group of $x(t)$.
    Let $\overline{t} \ra T'$ be a geometric point, put $y = x(\overline{t})$, and let $J(y) = \{j \ | \ y_j = 0 \}$.
    Write $g(\overline{t}) = [\epsilon] \in G_y$.
    Since $[\epsilon]$ fixes $y \in Y$, there exists $c \in \{\pm 1 \}$ such that $\epsilon_i = c$ for $i \not\in J(y)$.
    The homomorphism on the geometric stabilizer groups induced by $\Phi$ is
    \begin{equation}\label{hom_geom_stab_groups}
      \theta_y \colon G_y \ra \mu_2^{J(y)}, \ \ [\epsilon] \mapsto \left( \frac{\epsilon_j}{c} \right)_{j \in J(y)}.
    \end{equation}
    Indeed, in the affine coordinates $\zeta_j = y_j/y_i$ with $i \not\in J(y)$, one has $\zeta_j \mapsto (\epsilon_j/c)\zeta_j$, and $\Phi$ identifies $\zeta_j$ with the corresponding root coordinate.
    The homomorphism $\theta_y$ is injective, because, if its value on $[\epsilon]$ is trivial, then $\epsilon_j = c$ also for $j \in J(y)$.
    Thus all entries of $\epsilon$ equal $c$, so $\epsilon$ is diagonal and $[\epsilon] = 1 \in G_y$.
    We can also check that $\theta_y$ is surjective, so $\theta_y$ is actually an isomorphism.

    Since $\Phi(\alpha) \in \Aut_{\cU_n}(\Phi(\xi))$ is the identity, we have $\theta_y(g(\overline{t})) = 1$.
    Hence $g(\overline{t}) = 1$ for every geometric point $\overline{t} \ra T'$.
    Since $G$ is finite \'{e}tale, the inverse image under $g$ of the identity section is open and closed.
    It contains every geometric point of $T'$, and therefore equals $T'$.
    Thus $g = 1$, and so $\alpha_{T'}$ is the identity.
    By using fppf descent, we get $\alpha = \id_{\xi}$.
    Therefore, the automorphism homomorphism \zcref{autom_hom} has trivial kernel for every $T$ and $\xi$, and thus $\Phi$ is representable by algebraic spaces.
    
    Finally, we will check the \'{e}taleness of the morphism $\Phi$.
    The quotient atlas $p \colon Y \ra [Y/G]$ is a representable, surjective, finite \'{e}tale morphism.
    Also, by Proposition \zcref{lift_of_f_Y}, the morphism $\tilde{f}_Y$ is representable, finite, and \'{e}tale.
    Since $\Phi \circ p = \tilde{f}_Y$, $\Phi$ is representable by algebraic spaces, and \'{e}taleness is local on the source for the \'{e}tale topology (cf. \cite[\href{https://stacks.math.columbia.edu/tag/06M2}{Tag 06M2}]{stacks-project}), $\Phi$ is \'{e}tale. 
\end{proof}

\begin{Thm}\label{root=quotient}
    There is a natural equivalence of Deligne--Mumford stacks
    \begin{equation*}
        \Phi \colon [Y/G] \xrightarrow{\sim} \cU_n.
    \end{equation*}
    Under this equivalence, $\tilde{f}_Y \colon Y \ra \cU_n$ is a surjective, finite \'{e}tale, and representable morphism.	
\end{Thm}
\begin{proof}
We claim that $\Phi$ is an isomorphism \'{e}tale locally.
Then, by the \'{e}tale descent, $\Phi$ is an isomorphism globally.
To prove this claim, we use Lemma \zcref{criterion}.
As we checked in Lemma \zcref{separated_of_finite_type}, both $[Y/G]$ and $\cU_n$ are separated Deligne--Mumford stacks of finite type over $\bC$.
The morphism $\Phi \colon [Y/G] \ra \cU_n$ fits in the following diagram

 \[\begin{tikzcd}
	{[Y/G]} && {\mathcal{U}_n} \\
	\\
	Y/G && U.
	\arrow["\Phi", from=1-1, to=1-3]
	\arrow["{\pi_{[Y/G]}}"', from=1-1, to=3-1]
	\arrow["{\pi_n}", from=1-3, to=3-3]
	\arrow["\sim", from=3-1, to=3-3]
    \end{tikzcd}\]
By Lemma \zcref{ealeness_geometric_stabilizers}, the morphism $\Phi$ is \'{e}tale.
Moreover, from the proof of this lemma, we also know that the induced map \zcref{hom_geom_stab_groups} of geometric stabilizer groups at each closed point is bijective.
Therefore by Lemma \zcref{criterion}, for each closed point $x \in |[Y/G]|$, there exists an \'{e}tale open neighborhood $V \subseteq Y/G \simeq U$ of $\pi_{[Y/G]}(x)$ such that $V \times_{Y/G} [Y/G] \simeq V \times_U \cU_n$.
This completes the proof of this theorem.
\end{proof}

\subsection{Mixed braid groups and Orbifold fundamental groups}\label{mixed_braid_orbifold_section}
Note that
\begin{equation*}
    U \setminus B^{\circ} = \bP^n \setminus (\Delta \cup B) \simeq \UConf_n(\bP^1_{\bC} \setminus \{\lambda_0, \dots, \lambda_{n+2} \}).
\end{equation*}
Here, $\UConf_n(S)$ is the unordered configuration space of $n$ distinct points on $S$.
Let us denote $\pi_1(U\setminus B^{\circ}) = \pi_1 (\UConf_n(\bP^1_{\bC} \setminus \{\lambda_0, \dots, \lambda_{n+2} \}))$ by $SB_{n+3, n}$, and call it the {\it spherical mixed braid group}.
It is isomorphic to the mixed braid group $B_{n+2, n} := \pi_1(\UConf_n(\bC \setminus \{\lambda_1, \dots, \lambda_{n+2} \}))$ by identifying $\lambda_0 = \infty$ after applying a projective transformation.
Let $a_j$, $(j = 0, 1, \dots, n+2)$ be the meridian about the fixed puncture $\lambda_j$, and let $\sigma_i$, $(i = 1, \dots, n-1)$ be a standard half-twist.

The ramified $G$-covering map $f_Y \colon Y \ra U$ gives us the following short exact sequence
\begin{equation}\label{fYseq}
    1 \ra \pi_1(X \setminus (W \cup R)) = \pi_1(Y \setminus R^{\circ}) \ra SB_{n+3, n} \xrightarrow{\mu} G \ra 1. 
\end{equation}
Here, $\mu \colon SB_{n+3, n} \ra G$ is a monodromy map that satisfies
\begin{equation*}
    \mu(a_j) = \eta_j, \ \  \mu(\sigma_i) = 1_G,
\end{equation*}
where $\eta_j \in G$ is the class of the sign vector whose $j$th entry is $-1$ and whose other entries are $1$. 

Let $\pi_1^{\orb}(\cU_n)$ be the orbifold fundamental group of $\cU_n$.
Since $\pi_n \colon \cU_n \ra U$ is a coarse moduli space, and taking a $2$-root along a normal-crossing divisor imposes the square relation on the corresponding meridian \cite[Section 1.1]{Eys18}, we get 
\begin{equation}\label{orbifold_fundamental_group}
    \pi_1^{\orb}(\cU_n) \simeq SB_{n+3, n}/ \langle \langle a_0^2, \dots, a_{n+2}^2 \rangle \rangle_{\cB}.
\end{equation}
Here, $ \langle \langle a_0^2, \dots, a_{n+2}^2 \rangle \rangle_{\cB}$ is the normal closure of $\{a_0^2, \dots, a_{n+2}^2  \}$ within $SB_{n+3, n}$.

By Theorem \zcref{root=quotient}, $\tilde{f}_Y \colon Y \ra \cU_n$ is the surjective finite \'{e}tale morphism and gives us the following short exact sequence
\begin{equation}\label{tildeseq}
    1 \ra \pi_1(Y) \ra \pi_1^{\orb}(\cU_n) \xrightarrow{\tilde{\mu}} G \ra 1,
\end{equation}
where $\tilde{\mu}(a_j) = \eta_j$ and $\tilde{\mu}(\sigma_i) = 1_G$.

\begin{Cor}\label{ker_description_of_pi_1}
For every $n \geq 1$,
    \begin{equation*}
        \pi_1(X \setminus W) \simeq \ker (SB_{n+3, n}/ \langle \langle a_0^2, \dots, a_{n+2}^2 \rangle \rangle_{\cB} \xrightarrow{\tilde{\mu}} G).
    \end{equation*}
\end{Cor}
\begin{proof}
The statement follows from \zcref{orbifold_fundamental_group} and \zcref{tildeseq}.
\end{proof}

\subsection{The Meridian-Quotient Construction}

The short exact sequences \zcref{fYseq} and \zcref{tildeseq} fit into the following commutative diagram:

\[\begin{tikzcd}
	&& 1 && 1 &&&& \\
	\\
	&& \ker(\varphi_Y) && \ker(\varphi_U) \\
	\\
	1 && \pi_1(Y \setminus R^{\circ}) &&   \pi_1(U \setminus B^{\circ}) = SB_{n+3, n} && G && 1 \\
	\\
	1 && \pi_1(Y) && \pi_1^{\orb}(\cU_n) \simeq SB_{n+3, n}/ \langle \langle a_0^2, \dots, a_{n+2}^2 \rangle \rangle_{\cB} && G && 1 \\
	\\
	&& 1 && 1
	\arrow[from=1-3, to=3-3]
	\arrow[from=1-5, to=3-5]
	\arrow[from=3-3, to=3-5]
	\arrow[from=3-3, to=5-3]
	\arrow[from=3-5, to=5-5]
	\arrow[from=5-1, to=5-3]
	\arrow[from=5-3, to=5-5]
	\arrow["{\varphi_Y}"', from=5-3, to=7-3]
	\arrow["\mu", from=5-5, to=5-7]
	\arrow["{\varphi_U}"', from=5-5, to=7-5]
	\arrow[from=5-7, to=5-9]
	\arrow[from=5-7, to=7-7]
	\arrow[from=7-1, to=7-3]
	\arrow[from=7-3, to=7-5]
	\arrow[from=7-3, to=9-3]
	\arrow["{\tilde{\mu}}", from=7-5, to=7-7]
	\arrow[from=7-5, to=9-5]
	\arrow[from=7-7, to=7-9]
\end{tikzcd}\]

\begin{Prop}
    \begin{enumerate}
        \item[(1)] \begin{equation*}
                     \ker (\varphi_Y) \simeq \ker (\varphi_U) = \langle\langle a_0^2, \dots, a_{n+2}^2 \rangle \rangle_{\cB}.
                   \end{equation*}
        \item[(2)] \begin{equation*}
                     \pi_1(X \setminus W) = \pi_1(Y) \simeq \ker(\mu)/ \langle\langle a_0^2, \dots, a_{n+2}^2 \rangle\rangle_{\cB}.
                   \end{equation*}
    \end{enumerate}
\end{Prop}
\begin{proof}
     (1) $\ker (\varphi_U) = \langle\langle a_0^2, \dots, a_{n+2}^2 \rangle \rangle_{\cB}$ follows from \zcref{orbifold_fundamental_group}.
     Since $a_j^2 \in \ker(\mu)$ for all $j$, $\langle\langle a_0^2, \dots, a_{n+2}^2 \rangle \rangle_{\cB} \subset \ker (\mu)$.
     From the above commutative diagram, $\ker(\varphi_Y) = \ker(\mu) \cap \langle\langle a_0^2, \dots, a_{n+2}^2 \rangle \rangle_{\cB} = \langle\langle a_0^2, \dots, a_{n+2}^2 \rangle \rangle_{\cB}$.
    Statement (2) follows from (1).
\end{proof}

\section{Mixed Braid Groups and Reidemeister--Schreier Rewriting}\label{Section_Reidemister_Schreier}
In this section, we describe $\ker(\tilde{\mu} \colon \pi_1^{\orb}(\cU_n) \ra G)  \simeq \pi_1(Y) = \pi_1(X \setminus W)$ by their generators and relations.
To do so, we apply the {\it{Reidemeister--Schreier Method}}.
This method gives us the explicit description of $\pi_1(X \setminus W)$ (Theorem \zcref{Explicit_description_of_fundamental_group}).

\subsection{A Presentation of the mixed Braid groups due to Lambropoulou}\label{presentation}
Here, we introduce a presentation of the mixed braid groups $B_{n+2, n}$ due to Lambropoulou \cite{Lam00}.
When $n = 1$, we have $B_{3,1} = \pi_1(\bC \setminus\{\lambda_1, \lambda_2, \lambda_3 \}) \simeq F_3$, which is a free group of rank $3$.

Assume $n \geq 2$, and put $m := n + 2$.
We identify $\lambda_0$ with $\infty$.
After an orientation-preserving homeomorphism fixing $\infty$, we may choose a topological model with base configuration $\bz := \{z_1, \dots, z_n \}$, in which
\begin{equation*}
    \lambda_1 < \cdots < \lambda_m < z_1 < \cdots < z_n
\end{equation*}
lie on the real axis in $\bC$.
Recall that
\begin{equation*}
    B_{m, n} = B_{n+2, n} := \pi_1(\UConf_n(\bC \setminus \{\lambda_1, \dots, \lambda_{n+2} \}), \bz).
\end{equation*}
Via the standard configuration space identification, adjoining the $m$ fixed strands at $\lambda_1, \dots, \lambda_m$ identifies this group with Lambropoulou’s subgroup $B_{m, n} \subset B_{m+n}$, where $B_{m, n}$ is regarded as a braid group with $n$ moving strands and $m$ fixed strands.

We now specify based representatives of the meridians $a_j$ introduced in \zcref{mixed_braid_orbifold_section}.
Read path products and braid words chronologically from left to right.
Let us use upper semicircular access paths from $z_1$ to $\lambda_1, \dots, \lambda_m, z_2, \dots, z_n$, and the upward ray to $\infty$.
Separate these paths slightly near $z_1$, preserving their isotopy classes, so that their distinct initial directions have counter-clockwise cyclic order
\begin{equation*}
    \infty, \lambda_1, \dots, \lambda_m, z_2, \dots, z_n.
\end{equation*}
Finite meridians are counter-clockwise, and the meridian at $\infty$ is positive in the local coordinate $w = 1/z$ of $\bP^1_{\bC}$.
The representative of $a_j$ for $j = 0, \dots, m$ moves only $z_1$ around $\lambda_j$ along the chosen access path.
Let $\sigma_i$ be the counter-clockwise half-twist exchanging $z_i$ and $z_{i+1}$.

Let $\tau_1, \dots, \tau_{m+n-1}$ denote the counter-clockwise Artin generators of $B_{m+n}$.
Under the above inclusion $B_{m,n} \subset B_{m+n}$, 
\begin{align*}
 a_j&=\tau_m\cdots\tau_{j+1} \tau_j^2 \tau_{j+1}^{-1}\cdots \tau_m^{-1},  &&(1 \leq j \leq m),\\
 \sigma_i&=\tau_{m+i} &&(1 \leq i \leq n-1).
\end{align*}
These coincide with Lambropoulou's generators $a_{j, m+1}$ and $\sigma_{m+i}$ respectively in \cite[Section 1]{Lam00}.
Therefore, \cite[Theorem 3]{Lam00} gives an affine presentation with generators
\begin{equation*}
    a_1, \dotsm, a_m, \ \ \sigma_1, \dots, \sigma_{n-1},
\end{equation*}
with relations
\begin{align}
 \sigma_i\sigma_j&=\sigma_j\sigma_i &&(1 \leq i < j \leq n-1, \ j-i > 1),\tag{A1}\label{rel:A1}\\
 \sigma_i\sigma_{i+1}\sigma_i&=\sigma_{i+1}\sigma_i\sigma_{i+1}
 &&(1\le i\le n-2),\tag{A2}\label{rel:A2}\\
 a_k\sigma_i&=\sigma_i a_k &&(1\le k\le m,\ 2\le i\le n-1),\tag{M1}\label{rel:M1}\\
 \sigma_1a_k\sigma_1a_k&=a_k\sigma_1a_k\sigma_1
 &&(1\le k\le m),\tag{M2}\label{rel:M2}\\
 a_k(\sigma_1a_\ell\sigma_1^{-1})
 &=(\sigma_1a_\ell\sigma_1^{-1})a_k
 &&(1 \leq \ell<k \leq m).\tag{M3}\label{rel:M3}
\end{align}

For $2 \leq s \leq n$, put
\begin{equation*}
   P_m := a_1 \cdots a_m, \ \ \  A_{1s} := \sigma_{s-1}\cdots \sigma_2 \sigma_1^2 \sigma_2^{-1} \cdots \sigma_{s-1}^{-1}, \ \ \ \Omega_n := A_{12}A_{13}\cdots A_{1n}.
\end{equation*}
With $z_2, \dots, z_n$ fixed, the positive meridian around $z_s$ represents $A_{1s}$.
By using the above fixed system of base paths, the spherical relation is
\begin{equation}\label{spherical_relation}
    a_0P_m\Omega_n = 1, \ \ \ a_0 = (P_m\Omega_n)^{-1}.
\end{equation}

Let us consider \zcref{orbifold_fundamental_group}, i.e.,
\begin{equation*}
    \pi_1^{\orb}(\cU_n) \simeq SB_{m+1, n}/ \langle \langle a_0^2, \dots, a_m^2 \rangle \rangle_{\cB}.
\end{equation*}
The affine presentation of $\pi_1^{\orb}(\cU_n)$ is obtained from \zcref{rel:A1}--\zcref{rel:M3} by adding square relations
\begin{equation}\label{square_relations}
    a_k^2 = 1 \ \ (1 \leq k \leq m), \ \ \ \tag{SQ}
\end{equation}
and the {\it{infinity-square relation}}
\begin{equation}\label{infinity_square}
    (P_m \Omega_n)^2 = 1\tag{SQ0}.
\end{equation}
The last relation is the eliminated form of $a_0^2 = 1$.

\begin{Lem}\label{number_of_relators}
    Let $\cR$ be the ambient relator list consisting of \zcref{rel:A1}--\zcref{rel:M3} and \zcref{square_relations}--\zcref{infinity_square}.
    Then, 
    \begin{equation*}
        |\cR| = 2n^2 + 2n + 3.
    \end{equation*}
\end{Lem}
\begin{proof}
    By considering each index range, we get
    \begin{equation*}
    \begin{split}
         |\cR| &= \frac{(n-2)(n-3)}{2} + (n-2) + m(n-2) + m + \frac{m(m-1)}{2} + m + 1 \\
                &= 2n^2 + 2n + 3.
    \end{split}
    \end{equation*}
\end{proof}

\subsection{Reidemeister--Schreier Method}
Here, we recall the Reidemeister--Schreier method.
Given a presentation of a group $G$, this method shows how to compute a presentation for a subgroup $H$ of $G$.
The reference here is \cite[Chapter I\hspace{-1.2pt}I, Section 4]{LS77}.

\begin{Def}
    A Schreier transversal of a subgroup $H$ of $F$, free with basis $\cA$, is a subset $T$ of $F$ such that, for distinct $t$ in $T$, the cosets $Ht$ are distinct, that the union of $Ht$ is $F$, and such that every initial segment of an element of $T$ itself belongs to $T$.
\end{Def}

\begin{Prop}[Reidemeister--Schreier method, {\cite[Proposition 4.1]{LS77}}]\label{RS_method}
    Let $G = \langle \cA|R \rangle = F/N$, $F$ free with basis $\cA$ and $N$ the normal closure of $R$ in $F$, and let $\phi$ be the canonical map from $F$ onto $G$.
    Let $H$ be a subgroup of $G$, with $\hat{H}$ the inverse image of $H$ under $\phi$, and $T$ a Schreier transversal for $\hat{H}$ in $F$.
    For $w$ in $F$ we define $\overline{w}$ by the condition that 
    \begin{equation*}
        \hat{H}w = \hat{H}\overline{w}, \ \ \overline{w} \in T.
    \end{equation*}
    For $t \in T$ and $x \in \cA$ we define 
    \begin{equation*}
        \gamma(t, x) = tx(\overline{tx})^{-1}, \ \gamma(t, x^{-1}) = tx^{-1}(\overline{tx^{-1}})^{-1} = \gamma(tx^{-1}, x)^{-1}.
    \end{equation*}
    Then $H$ has a presentation $H = \langle \cA_1^* | R_1^* \rangle$, as follows.

    Let $\cA_1$ consist of all elements $\gamma(t, x) \neq 1$ for $t \in T$, $x \in \cA$, let $\cA_1^*$ be a set of elements $\gamma(t, x)^*$ in one-to-one correspondence with those of $\cA_1$, and let $F_1$ be the free group with basis $\cA_1^*$.
    Define a function $\tau$ from $F$ to $F_1$ as follows:\\
    if $w = y_1\cdots y_n$, $y_i \in \cA \cup \cA^{-1}$, then
    \begin{equation*}
        \tau(w) = \gamma(1, y_1)^*\cdots\gamma(\overline{y_1\cdots y_{i-1}}, y_i)^*\cdots\gamma(\overline{y_1\cdots y_{n-1}}, y_n)^*.
    \end{equation*}
    Then $R^*_1$ consists of all $\tau(trt^{-1})$ for $t \in T$ and $r \in R$.
\end{Prop}
In the following sections, we apply this method to the group $\pi_1^{\orb}(\cU_n)$ and its subgroup $\ker(\tilde{\mu}) \simeq \pi_1(Y)$.

\subsection{The Schreier transversal and the Schreier generators}
Let $F := F(a_1, \dots, a_m, \sigma_1, \dots, \sigma_{n-1})$ be a free group on a basis $\cA := \{a_1, \dots, a_m, \sigma_1, \dots, \sigma_{n-1} \}$ with $m = n+2$, and let $\hat{\mu} \colon F \ra G$ be the composite of the presentation map $F \ra \pi_1^{\orb}(\cU_n)$ and $\tilde{\mu} \colon \pi_1^{\orb}(\cU_n) \ra G$.
Notice that $\hat{\mu}$ sends $a_j$ to $\eta_j$ and every $\sigma_i$ to $1_G$, where $\eta_j \in G$ is the class of the sign vector with its only negative entry in position $j$.
Put $\hat{H} := \ker(\hat{\mu})$ and $H:= \ker(\tilde{\mu})$.
Thus, $\hat{H}$ is a free subgroup of $F$, whereas $H$ is its quotient by the ambient relators.

Let $S := \{1, \dots, m \}$.
For $I = \{i_1< \cdots < i_s \}  \subset S$, set
\begin{equation*}
    a_I := a_{i_1}\cdots a_{i_s}, \ a_{\emptyset} := 1.
\end{equation*}
Then, we can choose a Schreier transversal 
\begin{equation*}
    T := \{a_I \ | \ I \subset S \}, \ \ |T| = 2^m
\end{equation*}
for $\hat{H}$.

\begin{Lem}\label{state_I_lemma}
    For $w \in F$, let $e_j(w)$ be the exponent sum of $a_j$ in $w$ and put
    \begin{equation*}
        I(w) = \{j \in S \ | \ e_j(w) \equiv 1  \pmod 2 \}.
    \end{equation*}
Then
\begin{equation*}
    \hat{\mu}(w) = \prod_{j \in I(w)} \eta_j, \ \ \bar{w} = a_{I(w)}, \ \ \hat{H}w = \hat{H} a_{I(w)}.
\end{equation*}
In particular, $w \in \hat{H}$ if and only if $I(w) = \emptyset$.
Furthermore, 
\begin{equation*}
    I(vw) = I(v) \triangle I(w), \ \ I(w^{-1}) = I(w).
\end{equation*}
    Here, $I(v) \triangle I(w)$ is the symmetric difference
    \begin{equation*}
        I(v) \triangle I(w) := (I(v) \setminus I(w)) \cup (I(w) \setminus I(v)).
    \end{equation*}
    A letter $\sigma_i^{\pm}$ leaves the state $I$ unchanged, whereas $a_j^{\pm}$ changes $I$ to $I \triangle \{j \}$.
\end{Lem}
\begin{proof}
    We first compute the image $\hat{\mu}(w)$ of a word $w \in F$.
    For each $j$, the exponent sum $e_j(w)$ counts the occurrences of $a_j$ with sign $+1$ and those of $a_j^{-1}$ with sign $-1$.
    Since $G$ is abelian and the $\sigma_j$ have trivial image, we may collect the contributions of each $a_j$ in $G$, obtaining
    \begin{equation*}
        \hat{\mu}(w) = \prod_{j = 1}^m \eta_j^{e_j(w)}.
    \end{equation*}
    As $\eta_j^2 = 1_G$, the factor $\eta_j^{e_j(w)}$ depends only on the parity of $e_j(w)$.
    More precisely, it is $1_G$ if $e_j(w)$ is even, and $\eta_j$ if $e_j(w)$ is odd.
    Therefore
    \begin{equation*}
        \hat{\mu}(w)= \prod_{j \in I(w)} \eta_j.
    \end{equation*}

    We next identify the chosen representative of the right coset containing $w$.
    Write
    \begin{equation*}
        I(w) = \{i_1 < \cdots < i_s \}.
    \end{equation*}
    By definition, $a_{I(w)} = a_{i_1}\cdots a_{i_s}$ with $a_{\emptyset} =1$, and hence
    \begin{equation*}
        \hat{\mu}(a_{I(w)}) = \prod_{j \in I(w)}\eta_j = \hat{\mu}(w).
    \end{equation*}
    It follows that $w a_{I(w)}^{-1} \in \ker(\hat{\mu}) = \hat{H}$.
    Thus $w = h a_{I(w)}$ for some $h \in \hat{H}$, so
    \begin{equation*}
        \hat{H}w = \hat{H} a_{I(w)}.
    \end{equation*}
    Since $a_{I(w)}$ belongs to the chosen transversal $T$, it is the representative denoted by $\bar{w}$.
    Consequently, $\bar{w} = a_{I(w)}$.
    Moreover, the independence of $\eta_1, \dots,\eta_m$ implies that $\prod_{j \in I(w)} \eta_j = 1_G$ if and only if $I(w) = \emptyset$.
    Therefore $w \in \hat{H}$ if and only if $I(w) = \emptyset$.

    Finally, exponent sums satisfy
    \begin{equation*}
        e_j(vw) = e_j(v) + e_j(w), \ \ e_j(w^{-1}) = - e_j(w).
    \end{equation*}
    The first sum is odd precisely when exactly one of its summands is odd.
    Hence
    \begin{equation*}
        I(vw) = I(v) \triangle I(w).
    \end{equation*}
    Negating an integer does not change its parity, so $I(w^{-1}) = I(w)$.
    In particular, $I(a_j^{\pm1}) = \{ j\}$ and $I(\sigma_i^{\pm1}) = \emptyset$.
    Thus, when a word is read from left to right, adding $a_j^{\pm1}$ to the state $I$ changes the state to $I\triangle \{ j\}$, whereas adding $\sigma_i^{\pm1}$ leaves the state unchanged.
\end{proof}

Set 
\begin{equation*}
    I \triangle j := I \triangle\{j \} = 
        \begin{cases}
        I\cup \{j \}   &   \text{if}\ j \not\in I   \\
        I \setminus \{j \}         &   \text{if}\ j \in I,
    \end{cases}
\end{equation*}
and define 
\begin{equation}\label{schreier_basis}
    a_{I, j} := \gamma(a_I, a_j) = a_I a_j (a_{I \triangle j})^{-1}, \ \sigma_{I, i} := \gamma(a_I, \sigma_i) = a_I \sigma_i a_I^{-1}.
\end{equation}
as the Schreier generators.
Let $F_1$ be the free group on the non-identity Schreier generators corresponding to $a_{I,j}$ and $\sigma_{I, i}$.

Reidemeister--Schreier method (Proposition \zcref{RS_method}) tells us that the non-identity elements $a_{I, j}, \sigma_{I, i}$ form a free basis of $\hat{H}$.
We henceforth identify $F_1$ with $\hat{H}$ by identifying each abstract Schreier generator with its defining word in $F$.

\subsection{The rewriting map}
In this section, we define the rewriting map
\begin{equation*}
    \tau_I \colon F \ra F_1,
\end{equation*}
where $F_1$ is the free group on the Schreier generators $a_{I,j}$ and $\sigma_{I, i}$ in \zcref{schreier_basis}.

The standard rewriting map
\begin{equation*}
    \tau \colon F \ra F_1
\end{equation*}
in Proposition \zcref{RS_method} reads a word from left to right, starting in the identity coset $\hat{H}$, whose state is $\emptyset$.
We now introduce a version that starts in an arbitrary coset state.
Its purpose is to compute and rewrite the conjugate relators required by that proposition without repeatedly writing the surrounding transversal words.

Let $\cR$ be the relators list \zcref{rel:A1}--\zcref{rel:M3} and \zcref{square_relations}--\zcref{infinity_square} of the ambient presentation in \zcref{presentation}.
Proposition \zcref{RS_method} requires us to rewrite $a_I r a_I^{-1}$ for every $I \subset S$ and $r \in \cR$.
Since every ambient relator $r \in \cR$ belongs to $\hat{H}$, Lemma \zcref{state_I_lemma} gives us $I(r) = \emptyset$.
Reading this conjugate from the identity coset therefore gives us the state sequence
\begin{equation*}
    \emptyset \xrightarrow{a_I} I \xrightarrow{r} I \xrightarrow{a_I^{-1}} \emptyset.
\end{equation*}
Thus, although the entire conjugate is read from state $\emptyset$, its middle part $r$ is read from state $I$.
Let us denote the rewriting of this middle part by $\tau_I(r)$ that corresponds to $\tau(a_I r a_I^{-1})$.
More generally, we shall define $\tau_I(w)$ for any word $w \in F$.
The subscript $I$ specifies the state immediately before reading $w$.

\subsubsection{Definition of the rewriting map $\tau_I$}
Fix $I \subset S$, and write
\begin{equation*}
    w = y_1\cdots y_N, \ \ y_k \in \cA \cup \cA^{-1}.
\end{equation*}
Let
\begin{equation*}
    p_0 = 1, \ \ p_k = y_1\cdots y_k \ \ (1 \leq k \leq N)
\end{equation*}
be its prefixes.
Define the state after reading the first $k$ letters, starting from $I$, by
\begin{equation*}
    I_k := I \triangle I(p_k).
\end{equation*}
In particular, $I_0 = I$ and $I_N = I \triangle I(w)$.

The procedure reads the letters from left to right.
At each step, it updates the state and outputs one factor according to the following table, where $J$ denotes the state immediately before reading the letter:
\[
\renewcommand{\arraystretch}{1.3}
\begin{array}{c|c|c|c}
\text{Input letter}
&
\text{State before reading the letter}
&
\text{State after reading the letter}
&
\text{Output factor}
\\ \hline
a_j & J & J\triangle j & a_{J,j}\\
a_j^{-1} & J & J\triangle j & a_{J\triangle j,j}^{-1}\\
\sigma_i & J & J & \sigma_{J,i} \\
\sigma_i^{-1} & J & J & \sigma_{J,i}^{-1}
\end{array}
\]

Explicitly, the factor assigned to the $k$-th letter is 

\begin{equation*}
  \rho_k :=
  \begin{cases}
    a_{I_{k-1}, j}, & \text{if $y_k = a_j$,} \\
    a_{I_k, j}^{-1},              & \text{if $y_k = a_j^{-1}$,} \\
    \sigma_{I_{k-1}, i},       & \text{if $y_k = \sigma_i$,} \\
    \sigma_{I_{k-1}, i}^{-1},  & \text{if $y_k = \sigma_i^{-1}$}.
  \end{cases}
\end{equation*}

We define 
\begin{equation}
    \tau_I(w) := \rho_1 \cdots \rho_N,
\end{equation}
where identity factors are deleted and the product is freely reduced in the Schreier generators.
The empty input word gives the empty output word.

For example, reading $a_j^2$ from the state $I$ gives the state sequence
\begin{equation*}
    I \xrightarrow{a_j} I \triangle j \xrightarrow{a_j} I,
\end{equation*}
and hence the output is $\tau_I(a_j^2) = a_{I, j} a_{I\triangle j, j}$.
For $r = a_j \sigma_i a_j^{-1} \sigma_i^{-1}$, the state sequence is 
\begin{equation*}
    I \xrightarrow{a_j} I \triangle j \xrightarrow{\sigma_i} I \triangle j \xrightarrow{a_j^{-1}} I \xrightarrow{\sigma_i^{-1}} I,
\end{equation*}
so $\tau_I(r) = a_{I, j} \sigma_{I \triangle j, i} a_{I, j}^{-1} \sigma_{I, i}^{-1}$.

\subsubsection{Relation with the standard rewriting map $\tau$}
Starting from the identity coset gives the standard procedure of Proposition \zcref{RS_method}. Therefore, $\tau= \tau_{\emptyset}$.

For $I \subset S$ and $w \in F$, consider the full word $a_I w a_{I \triangle I(w)}^{-1}$, and apply $\tau$, starting from state $\emptyset$.
The initial segment $a_I$ is an increasing transversal word.
Each of its prefixes is again a transversal word, so its output factors are all identities.
The middle segment $w$ starts from state $I$ and produces $\tau_I(w)$.
The final segment $a_{I \triangle I(w)}^{-1}$, read from state $I \triangle I(w)$, reverses the chosen transversal word and likewise contributes only identity factors.
Hence
\begin{equation}
    \tau_I(w) = \tau(a_I w a_{I \triangle I(w)}^{-1}).
\end{equation}
In particular, for $w \in \hat{H}$, $\tau_I(w) = \tau(a_I w a_I^{-1})$ as we expected.

\subsubsection{Concatenation and inversion formulas}
Suppose that $vw$ is the concatenation of two words $v, w \in F$ and read from state $I$.
After reading $v$, the state is $I \triangle I(v)$.
The word $w$ must therefore be rewritten from that state, giving 
\begin{equation}\label{concatenation_formula}
    \tau_{I}(vw) = \tau_I(v) \tau_{I \triangle I(v)} (w).
\end{equation}

For inversion, reading $w$ from $I$ ends at $I \triangle I(w)$.
Reading $w^{-1}$ from the state $I \triangle I(w)$ reverses the state sequence and outputs the inverse of the original factors in reverse order.
Hence $\tau_{I \triangle I(w)}(w^{-1}) = \tau_I (w) ^{-1}$.
Equivalently,
\begin{equation}\label{inversion_formula}
    \tau_I(w^{-1}) = \tau_{I \triangle I(w)} (w)^{-1}.
\end{equation}
Consequently, $\tau \colon F \ra F_1$ is not a homomorphism in general.

\subsection{The rewritten relations}
In this section, we give the rewritten relations.
Every ambient relator $r \in \cR$ has $I(r) = \emptyset$.
Hence, Proposition \zcref{RS_method} imposes precisely
\begin{equation}
    \tau(a_I r a_I^{-1}) = \tau_I(r) = 1 \ \ (I \subset S). 
\end{equation}
If a relation is written as $L = R$, its relator is $LR^{-1}$.
Since $I(L) = I(R)$, concatenation and inversion formulas \zcref{concatenation_formula}--\zcref{inversion_formula} show that its rewriting is equivalently $\tau_I(L) = \tau_I(R)$.
Therefore, for every $I \subset S$, the rewriting of \zcref{rel:A1}--\zcref{rel:M3} gives the following relations:
\begin{align}
 \sigma_{I,i}\sigma_{I,j}&=\sigma_{I,j}\sigma_{I,i} &&(1 \leq i < j \leq n-1, \ j-i > 1),\tag{RS-A1}\label{rel:RSA1}\\
 \sigma_{I, i}\sigma_{I, i+1}\sigma_{I,i}&=\sigma_{I, i+1}\sigma_{I,i}\sigma_{I, i+1} &&(1\le i\le n-2),\tag{RS-A2}\label{rel:RSA2}\\
 a_{I,k}\sigma_{I \triangle k, i}&=\sigma_{I, i} a_{I,k} &&(1\le k\le m,\ 2\le i\le n-1),\tag{RS-M1}\label{rel:RSM1}\\
 \sigma_{I,1} a_{I,k} \sigma_{I \triangle k, 1} a_{I \triangle k, k}&=a_{I,k}\sigma_{I \triangle k, 1} a_{I \triangle k, k} \sigma_{I, 1} &&(1\le k\le m),\tag{RS-M2}\label{rel:RSM2}\\
 a_{I, k}(\sigma_{I \triangle k, 1}a_{I \triangle k, \ell}\sigma_{I \triangle k \triangle \ell, 1}^{-1}) &=(\sigma_{I,1}a_{I, \ell}\sigma_{I \triangle \ell,1}^{-1})a_{I \triangle \ell,k} &&(\ell<k).\tag{RS-M3}\label{rel:RSM3}
\end{align}
The rewriting of the square relations \zcref{square_relations} gives
\begin{equation}\label{RS_square_relation}
    a_{I, k} a_{I \triangle k, k} = 1 \ \ \ (1 \leq k \leq m). \ \ \tag{RS-SQ}
\end{equation}
For the infinity-square \zcref{infinity_square}, let us consider the rewriting of $P_m$ and $\Omega_n$.
Define 
\begin{equation*}
    I^{(k)} := I \triangle \{1, \dots, k \}\ (0 \leq k \leq m), \ \ I^c := I^{(m)} = S \setminus I.
\end{equation*}
Let $P_I$ denote $\tau_I(P_m)$.
Then, by the concatenation formula \zcref{concatenation_formula}, we get
\begin{equation*}
    P_I := \tau_I(P_m) = \tau_I(a_1 \cdots a_m) = a_{I^{(0)}, 1} a_{I^{(1)}, 2} \cdots a_{I^{(m-1)}, m}.
\end{equation*}
Similarly, we put $A_{1s}(I) := \tau_I(A_{1s})$ and $\Omega_I :=\tau_I(\Omega_n)$.
Then 
\begin{equation}
    A_{1s}(I) := \tau_I(A_{1s}) = \tau_I(\sigma_{s-1}\cdots \sigma_2 \sigma_1^2 \sigma_2^{-1} \cdots \sigma_{s-1}^{-1}) = \sigma_{I, s-1}\cdots \sigma_{I, 2} \sigma_{I,1}^2 \sigma_{I,2}^{-1} \cdots \sigma_{I,s-1}^{-1}
\end{equation}
and 
\begin{equation*}
    \Omega_I := \tau_I(\Omega_n) = \prod_{s=2}^n A_{1s}(I).
\end{equation*}
Since the state sequence of $(P_m \Omega_n)^2$ is
\begin{equation*}
    I \xrightarrow{P_m} I^c \xrightarrow{\Omega_n} I^c \xrightarrow{P_m} I \xrightarrow{\Omega_n} I,
\end{equation*}
the concatenation formula \zcref{concatenation_formula} tells us that the infinity-square \zcref{infinity_square} becomes
\begin{equation}\label{RS_infinity_square}
    P_I \Omega_{I^c} P_{I^c} \Omega_I = 1\tag{RS-SQ0}.
\end{equation}

In summary, we obtain the following theorem.

\begin{Thm}\label{Explicit_description_of_fundamental_group}
    Assume $n \geq 2$.
    The fundamental group $\pi_1(X \setminus W)$ has a presentation with the non-identity free Schreier generators $a_{I, j}$ and $\sigma_{I, i}$
    \begin{equation*}
        a_{I,1}, \cdots, a_{I,n+2}, \sigma_{I, 1}, \cdots, \sigma_{I, n-1},
    \end{equation*}
  for every $I \subset S$ and all relations \zcref{rel:RSA1}--\zcref{RS_infinity_square}. \qed
\end{Thm}

Let us count the number of generators and relations of $\pi_1(X \setminus W)$.
\begin{Prop}\label{numbers_of_generators_relators}
    Assume $n \geq 2$.
    Before any further Tietze transformations, the number of generators of $\pi_1(X \setminus W)$ is $2^{n+3}n +1$,
    and the number of formal relators is $2^{n+2}(2n^2 + 2n + 3)$.
\end{Prop}
\begin{proof}
    First, let us count the non-identity Schreier generators.
    The Schreier generators are $a_{I,j} = a_I a_j a_{I\triangle j}^{-1}$ and $\sigma_{I, i} = a_I \sigma_i a_I^{-1}$, where $I \subset S$, $1 \leq j \leq m$, and $1 \leq i \leq n-1$.
    Since $|T| = 2^m$, there are initially {\it{$m2^m$ $a$-type Schreier generators}} $a_{I, j}$ and $(n-1)2^m$ {\it{$\sigma$-type Schreier generators $\sigma_{I, i}$}}.
    We must omit precisely those that equal $1$ in $F_1$.
    
    For an $a$-type generator, $a_{I, j} = 1$ in $F_1$ if and only if $a_I a_j = a_{I \triangle j}$ in $F$.
    Both sides of the equation $a_I a_j = a_{I \triangle j}$ are equal in $F$ if and only if their letters agree in order.
    If $j \in I$, then the word $a_Ia_j$ has length $|I|+ 1$, whereas the word $a_{I \triangle j}$ has length $|I| -1$, and therefore cannot be equal.
    If $j \notin I$, then the word $a_{I\triangle j}$ lists the elements of $I \cup \{j \}$ in increasing order.
    It agrees with $a_I a_j$ precisely when $j$ is larger than every element in $I$.
    Hence, $a_{I, j} = 1$ in $F_1$ if and only if $j \notin I$ and $j > \max I$, where $\max \emptyset = 0$.
    For a fixed $j$, this is equivalent to the condition $I \subset \{1, \dots, j-1 \}$, so there are exactly $2^{j-1}$ such words.
    Consequently, the total number of identity $a$-type generators is $\sum_{j = 1}^m 2^{j-1} = 2^m -1$.

    On the other hand, no $\sigma$-type generator is identity, i.e., $a_I \sigma_i \neq a_I$ in $F$ for every $I$ and $i$, since $a_I$ contains only $a$-type letters and $\sigma_i$ cannot be canceled out.
    Therefore, the number of the non-identity Schreier generators is
    \begin{equation*}
        (m2^m -(2^m-1)) + (n-1)2^m = 2^{n+3}n +1.
    \end{equation*}

    Next, we compute the number of formal relators.
    Proposition \zcref{RS_method} supplies one rewritten relator
    \begin{equation*}
        \tau(a_I r a_I^{-1}) = \tau_I(r) = 1
    \end{equation*}
    for each ambient relator $r \in \cR$ and each initial state $I \subset S$.
    Since there are $|\cR| = 2n^2 + 2n + 3$ choices of $r$ (Lemma \zcref{number_of_relators}) and $2^m$ choices of $I$, the number of formal relators is 
    \begin{equation*}
        2^m|\cR| = 2^{n+2}(2n^2 + 2n + 3).
    \end{equation*}
\end{proof}
In \zcref{n=2_example_section}, we apply Tietze transformations to get a minimal presentation of $\pi_1(X \setminus W)$ in the case of $n = 2$.

\section{Examples}\label{example_section}
In this section, we apply our results so far to the case $n = 1, 2, 3$.
Especially, in the case $n =2$, we give a purely algebraic minimal presentation of the fundamental group $\pi_1(X \setminus W)$.
\subsection{$n = 1$: elliptic curve}
\begin{Prop}
    For $n = 1$, $X$ is an elliptic curve in $\bP^3_{\bC}$, $W = \emptyset$, and
    \begin{equation*}
        \pi_1(X \setminus W) = \pi_1(X) \simeq H_1(X \setminus W, \bZ) \simeq \bZ^2.
    \end{equation*}
\end{Prop}
\begin{proof}
    This follows from Proposition \zcref{X_and_W} (1).
\end{proof}

\subsection{$n=2$: degree $4$ del Pezzo surface $dP_5$}\label{n=2_example_section}

\subsubsection{An algebraic presentation}
\begin{Prop}\label{n=2_general_statement}
    For $n = 2$, $X$ is a degree $4$ del Pezzo surface $dP_5$ in $\bP^4_{\bC}$, and $W$ is the union of $16$ $(-1)$-curves.
     The Reidemeister--Schreier presentation of $\pi_1(X \setminus W)$ has $65$ generators and $240$ formal relators.
    Its abelianization $H_1(X \setminus W, \bZ)$ is $\bZ^{10}$.
\end{Prop}
\begin{proof}
    The first statement follows from Proposition \zcref{X_and_W} (2).
    Proposition \zcref{numbers_of_generators_relators} tells us there are $2^5\cdot2 +1 = 65$ generators and $2^4\cdot15=240$ relators.
    The abelianization follows from Theorem \zcref{first_homology_group} (2).
\end{proof}

For $n = 2$, the ambient generators are $a_1, \dots, a_4$ and $\sigma_1$.
Write $m_I := \sigma_{I, 1} = a_I \sigma_1 a_I^{-1}$.
The elimination in \zcref{elimination} retains the meridians
\begin{equation}\label{eq:ten}
\begin{split}
(\widetilde g_1,\widetilde g_2,\widetilde g_3,\widetilde g_4,\widetilde g_5)
 &=(m_\varnothing,m_1,m_2,m_{12},m_3),\\
(\widetilde g_6,\widetilde g_7,\widetilde g_8,\widetilde g_9,\widetilde g_{10})
 &=(m_{13},m_{23},m_4,m_{14},m_{24}).
\end{split}
\end{equation}
We make the free basis change
\begin{equation}\label{eq:newten}
g_i=\widetilde g_i\ (i\ne5,8),\qquad
g_5=\widetilde g_1^{-1}\widetilde g_2\widetilde g_3\widetilde g_5,
\qquad g_8=g_5\widetilde g_8.
\end{equation}
Its inverse is
$\widetilde g_5=g_3^{-1}g_2^{-1}g_1g_5$ and
$\widetilde g_8=g_5^{-1}g_8$.
For simplicity, we put
\begin{equation}\label{eq:word_abbreviations}
U_1 := g_1g_4,\qquad U_2 := g_4g_6g_7,\qquad
U_3 := g_4g_6g_9,\qquad U_4 := g_4g_7g_{10}g_1.
\end{equation}

\begin{Prop}\label{algebraic_presentation}
    Assume $n = 2$.
    The group $\pi_1(X \setminus W)$ has the presentation
    \begin{equation}
        \pi_1(X \setminus W) \simeq \langle g_1, \dots, g_{10} \  | \ q_1, \dots, q_{25} \rangle,
    \end{equation}
    where the relators $q_i$ are listed in \zcref{relators_list}.
    The first twenty-one relators are $g_sxg_s^{-1}\Phi_s(x)^{-1}$ from \zcref{tab:actions}, ordered column by column, and the final four are
\begin{equation}\label{eq:four_relators}
\begin{split}
q_{22}:=[g_9,g_7]&=1,\qquad q_{23}:=[g_1,g_8]=1,\\
q_{24}:=[g_{10},g_9^{-1}g_6g_9]&=1,\qquad
q_{25}:=[g_4,g_8^{-1}U_2g_9g_{10}g_1]=1,
\end{split}
\end{equation}
Here $[x,y]=xyx^{-1}y^{-1}$.
\end{Prop}
\begin{table}[htbp]
\centering\renewcommand{\arraystretch}{1.1}
\caption{The twenty-one equations $g_sxg_s^{-1}=\Phi_s(x)$, for $s=2,3,5$.}
\label{tab:actions}
\begin{tabular}{@{}cccc@{}}
\toprule
$x$ & $\Phi_2(x)$ & $\Phi_3(x)$ & $\Phi_5(x)$\\
\midrule
$g_1$ & $g_1$ & $g_1$ & $g_1$\\
$g_4$ & $g_4$ & $g_4$ & $U_2g_4U_2^{-1}$\\
$g_6$ & $g_6$ & $(g_7U_1)g_6(g_7U_1)^{-1}$ & $U_2g_6U_2^{-1}$\\
$g_7$ & $(U_1g_6)g_7(U_1g_6)^{-1}$ & $g_7$ & $U_2g_7U_2^{-1}$\\
$g_8$ & $U_3g_8U_3^{-1}$ & $U_4g_8U_4^{-1}$ & $U_2g_8U_2^{-1}$\\
$g_9$ & $g_9$ & $(g_{10}U_1)g_9(g_{10}U_1)^{-1}$ & $g_8g_9g_8^{-1}$\\
$g_{10}$ & $(U_1g_9)g_{10}(U_1g_9)^{-1}$ & $g_{10}$ & $g_8g_{10}g_8^{-1}$\\
\bottomrule
\end{tabular}
\end{table}

\begin{proof}
    As explained in Proposition \zcref{n=2_general_statement}, the raw presentation of $\pi_1(X \setminus W)$ has $65$ generators and $240$ formal relators.
    The Tietze transformations detailed and verified in \zcref{Tietze_transformation_app_A} convert it into $\langle g_1, \dots, g_{10} \  | \ q_1, \dots, q_{25} \rangle$, with generators and relators specified above.
\end{proof}

Explicitly, the indexing of the first twenty-one relators is
\begin{equation}\label{eq:ordered_relators}
q_{7(k-1)+\ell}=g_{s_k}g_{j_\ell}g_{s_k}^{-1}
\Phi_{s_k}(g_{j_\ell})^{-1}\quad(1\le k\le3,\ 1\le\ell\le7),
\end{equation}
where $(s_1,s_2,s_3)=(2,3,5)$ and
$(j_1,\ldots,j_7)=(1,4,6,7,8,9,10)$.
\zcref{relators_list} lists the expanded words $q_1,\ldots,q_{25}$.

\subsubsection{Minimality of the presentation}
Here, we show the presentation given by Proposition \zcref{algebraic_presentation} is minimal.
Put $F_1 := F(g_1, \dots, g_{10})$.
Then, $F_1^{ab} := F_1/[F_1, F_1] \simeq \bZ^{10}$.
Let $\bar{g}_i \in F_1^{ab}$ be the image of $g_i$ under the abelianization map $\ab \colon F_1 \ra F_1^{ab}$.
For a word $w \in F_1$, let $e_i(w)$ be the exponent sum of $g_i$ in $w$.
Cancellation does not change this sum, and 
\begin{equation*}
    \ab(w) = \sum_{i = 1}^{10} e_i(w)\bar{g}_i.
\end{equation*}
If $w$ has exponent sum zero in every generator $g_i$, i.e., $e_i(w) = 0$ for all $i$, then $\ab(w) = 0$.
This means $w \in [F_1, F_1]$.
The weight-two case of Hall's basis theorem (\cite[Theorem 4.1]{Hall50}) gives an isomorphism
\begin{equation*}
    \theta \colon [F_1, F_1]/[F_1, [F_1, F_1]] \xrightarrow{\sim} \bigwedge^2(F_1/[F_1, F_1]), \ \ \overline{[g_i, g_j]} \mapsto \bar{g}_i \wedge \bar{g}_j.
\end{equation*}
Through $\theta$, we identify $[F_1, F_1]/[F_1, [F_1, F_1]]$ with $\bigwedge^2(F_1/[F_1, F_1]) \simeq \bZ^{45}$, and then the classes $E_{ij} := \bar{g}_i \wedge \bar{g}_j$ with $1\leq i < j \leq 10$ form a free basis of it.

\begin{Lem}\label{linearly_independence_lemma}
    Each of $q_1, \dots, q_{25}$ has exponent sum zero in every generator $g_j$, and hence belongs to $[F_1, F_1]$.
    Their images in $[F_1, F_1]/[F_1, [F_1, F_1]]$ are linearly independent and span a direct summand.
\end{Lem}
\begin{proof}
    We can directly check the first statement from the list in \zcref{relators_list}.

    For simplicity, put $M := [F_1, F_1]/[F_1, [F_1, F_1]]$.
    Write $\bar{q}_i$ for the image of $q_i$ in $M$.
    We will show that $\bar{q}_1, \dots, \bar{q}_{25}$ extend to a $\bZ$-basis of $M$.
    Set $S := \{2,3,5 \}$, $J := \{1,4,6,7,8,9,10 \}$, and $N_J := \oplus_{j \in J} \bZ \bar{g}_j$.
    Choose a basis of $M$ consisting of the twenty-one mixed vectors $E_{sj} = \bar{g}_s \wedge \bar{g}_j$ for $s \in S$ and $j \in J$, together with the vectors $E_{ab} = \bar{g}_a \wedge \bar{g}_b$ for $a < b$ whose indices both lie in $S$ or both lie in $J$.

    By \zcref{tab:actions}, each $\Phi_s(g_j)$ can be written as $\Phi_s(g_j) = c_{sj} g_j c_{sj}^{-1}$, where $c_{sj}$ involves only generators indexed by $J$.
    The corresponding relator $q_i$ is therefore
    \begin{equation*}
        q_i = g_s g_j g_s^{-1} \Phi_s(g_j)^{-1} = [g_s, g_j][c_{sj, g_j}]^{-1}
    \end{equation*}
    whose image in $M$ is
    \begin{equation*}
        \begin{split}
             \bar{q}_i &= \bar{g}_s \wedge \bar{g}_j - \ab(c_{sj}) \wedge \bar{g}_j\\
                       &= E_{sj} + \eta_{sj}, \ \ \ \eta_{sj} \in \wedge^2 N_J.
        \end{split}
    \end{equation*}
   Here, we set $\eta_{sj} := - \ab(c_{sj}) \wedge \bar{g}_j \in \wedge^2 N_J$ for simplicity.
   By the choice of $E_{sj}$, replacing the twenty-one mixed basis vectors $\{E_{sj} \}$ by the corresponding images $\bar{q}_1, \dots, \bar{q}_{21}$ still gives a $\bZ$-basis of $M$.

   Next, write $E_{ij} = \bar{g}_i\wedge\bar{g}_j$ for $i < j$.
   Then, the remaining four images are 
   \begin{equation*}
       \bar{q}_{22} = -E_{79}, \ \ \bar{q}_{23} = E_{18}, \ \ \bar{q}_{24} = - E_{6, 10}, 
   \end{equation*}
   and
   \begin{equation*}
       \bar{q}_{25} = -E_{14}+E_{46}+E_{47}-E_{48}+E_{49}+E_{4, 10}.
   \end{equation*}
   We may therefore replace the four basis vectors $E_{79}, E_{18}, E_{6, 10}, E_{14}$ by $\bar{q}_{22}, \bar{q}_{23}, \bar{q}_{24}, \bar{q}_{25}$ respectively.
   The resulting $\bZ$-basis of $M$ contains $\bar{q}_1, \dots, \bar{q}_{25}$.
   Consequently, their images are linearly independent and span a direct summand.
\end{proof}

For a finitely presented group $G$, let $d(G)_{min}$ be the least number of generators and let $s(G)_{min}$ be the least number of relators among all finite presentations.
We define its {\it{deficiency}} $\D(G)$ as the maximum over all presentations for $G$ of the number of generators minus the number of relators.
For simplicity, we put $\pi := \pi_1(X \setminus W)$.

\begin{Thm}\label{minimality_theorem}
    Assume $n = 2$.
    For $\pi := \pi_1(X \setminus W)$, we have
    \begin{equation*}
        H_1(\pi, \bZ) \simeq \bZ^{10}, \ \ H_2(\pi, \bZ) \simeq \bZ^{25}.
    \end{equation*}
    Moreover, every finite presentation of $\pi$ with $d$ generators and $s$ relators satisfies $d \geq 10$ and $s \geq d + 15$.
    Consequently,
    \begin{equation*}
        d(\pi)_{min} = 10, \ \ s(\pi)_{min} = 25, \ \  \D(\pi) = -15.
    \end{equation*}
    Thus, the presentation obtained by Proposition \zcref{algebraic_presentation} is minimal.
\end{Thm}
\begin{proof}
    Let $ N_1 := \langle \langle q_1, \dots, q_{25} \rangle \rangle_{F_1}$ be the normal closure of the relators in $F_1$.
    Proposition \zcref{algebraic_presentation} gives $\pi \simeq F_1/N_1$.
    By Lemma \zcref{linearly_independence_lemma}, each $q_i$ lies in $[F_1, F_1]$.
    Moreover, since $[F_1, F_1] \triangleleft F_1$, the whole normal closure $N_1$ lies in $[F_1, F_1]$.
    Thus,
    \begin{equation*}
        H_1(\pi, \bZ) = (F_1/N_1)^{ab} \simeq F_1/[F_1, F_1] \simeq \bZ^{10}.
    \end{equation*}

    For a presentation $\pi = F_1/N_1$, Hopf's formula (\cite[Chapter I\hspace{-1.2pt}I, Theorem 5.3]{Bro82}) computes the second group homology as
    \begin{equation*}
        H_2(\pi, \bZ) = H_2(F_1/N_1, \bZ) \simeq (N_1\cap[F_1, F_1])/[F_1, N_1] = N_1/[F_1, N_1].
    \end{equation*}
    So, let us compute $N_1/[F_1, N_1]$.
    Since $[N_1, N_1] \subset [F_1, N_1]$, the group $N_1/[F_1, N_1]$ is abelian.
    Therefore, there exists a surjection
    \begin{equation*}
        \alpha \colon \bZ^{25} \twoheadrightarrow N_1/[F_1, N_1], \ \ (a_1, \dots, a_{25}) \mapsto q_1^{a_1}\cdots q_{25}^{a_{25}} [F_1, N_1] .
    \end{equation*}
    The inclusion $N_1 \subset [F_1, F_1]$ and $[F_1, N_1] \subset [F_1, [F_1, F_1]]$ induce a homomorphism
    \begin{equation*}
        N_1/[F_1, N_1] \ra [F_1, F_1]/[F_1, [F_1, F_1]].
    \end{equation*}
    Since the classes of the $q_i$ generate its domain, its image is exactly $L := \sum_{i = 1}^{25} \bZ \bar{q}_i$.
    Denote the resulting surjection onto $L$ by $\beta$.
    We obtain
    \begin{equation*}
        \bZ^{25} \xrightarrow{\alpha} N_1/[F_1, N_1] \xrightarrow{\beta} L \hookrightarrow [F_1, F_1]/[F_1, [F_1, F_1]].
    \end{equation*}
    The composite $\beta \alpha$ sends the standard basis of $\bZ^{25}$ to $\bar{q}_1, \dots, \bar{q}_{25}$, which form a basis of $L$ by Lemma \zcref{linearly_independence_lemma}.
    Thus, $\beta\alpha$ is an isomorphism. 
    In particular, $\ker(\alpha) \subset \ker(\beta \alpha) = 0$.
    Thus, the surjective map $\alpha$ is actually an isomorphism.
    Therefore, 
    \begin{equation*}
        H_2(\pi, \bZ) \simeq N_1/[F_1, N_1] \simeq \bZ^{25}.
    \end{equation*}

    We now consider an arbitrary finite presentation
    \begin{equation*}
        \pi = \langle x_1, \dots, x_d \ | \ r_1, \dots, r_s \rangle.
    \end{equation*}
    The image of $x_1, \dots, x_d$ generate the abelianization $\pi^{ab} \simeq \bZ^{10}$, so there is a surjection $\bZ^d \twoheadrightarrow \bZ^{10}$.
    Hence $d \geq 10$.

    Let $X_{\pi}$ be a $2$-dimensional cell complex with $\pi_1(X_{\pi}) \simeq \pi$ (\cite[Corollary 1.28]{Hat02}).
    Here, $X_{\pi}$ is constructed from $\bigvee_{i = 1}^{d} S^1$ by attaching $2$-cells $e_j^2$ by the loops specified by the word $r_j$.
    Therefore $\chi(X_{\pi}) = 1 - d + s$.
    Put $b_i(X_{\pi}) := \dim_{\bQ}H_i(X_{\pi}, \bQ)$.
    Since $X_{\pi}$ is connected and two-dimensional, and $b_1(X_{\pi}) = \rank H_1(\pi, \bZ) = 10$, its Euler characteristic also satisfies
    \begin{equation*}
        \chi(X_{\pi}) = 1 - b_1(X_{\pi}) + b_2(X_{\pi}) = 1 - 10 + b_2(X_{\pi}).
    \end{equation*}
    Comparing these two expressions gives $b_2(X_{\pi}) = s - d + 10$.
    Hopf's exact sequence \cite[(0.1)]{Bro82} gives
    \begin{equation*}
        \pi_2(X_{\pi}) \ra H_2(X_{\pi}, \bZ) \ra H_2(\pi, \bZ) \ra 0,
    \end{equation*}
    and a surjection $H_2(X_{\pi}, \bQ) \twoheadrightarrow H_2(\pi, \bQ) \simeq \bQ^{25}$.
    Therefore, $b_2(X_{\pi}) \geq 25$.
    Together with the above equation, this gives $s-d+10 \geq 25$, and hence $s \geq d + 15 \geq 25$.
    Since the presentation is arbitrary, we get
    \begin{equation*}
         d(\pi)_{min} = 10, \ \ s(\pi)_{min} = 25, \ \  \D(\pi) \leq -15.
    \end{equation*}
    The presentation given by Proposition \zcref{algebraic_presentation} attains all three bounds, proving the statement.

\end{proof}

\subsection{$n = 3$: Fano threefold}

\begin{Prop}
    For $n = 3$, $W$ is irreducible of class $8H$ and has projective degree $32$ in $X$.
    The Reidemeister--Schreier presentation of $\pi_1(X \setminus W)$ has $193$ generators and $864$ formal relators, and Tietze transformations can transform it into a presentation with $4$ generators and $78$ relators.
    Its abelianization $H_1(X \setminus W, \bZ)$ is $\bZ/8\bZ$.
\end{Prop}
\begin{proof}
    By Proposition \zcref{X_and_W}, $W$ is irreducible of class $8H$ and has projective degree $8H^3 = 32$, since $H^3 = \deg (X) = 4$ in $\bP^5_{\bC}$.
    Proposition \zcref{numbers_of_generators_relators} tells us there are $3\cdot2^6 +1 = 193$ generators and $2^5\cdot27=864$ relators.
    For a presentation with $4$ generators and $78$ relators, the complete sequence and a verification program are given in \zcref{appndix_C}.
    The abelianization follows from Theorem \zcref{first_homology_group} (3). 
\end{proof}
We do not address the minimality of this presentation.

\appendix
\section{The algebraic calculation for $n = 2$}\label{Tietze_transformation_app_A}
Starting from the Reidemeister--Schreier presentation of Theorem \zcref{Explicit_description_of_fundamental_group}, this appendix\footnote{We used ChatGPT 6 Astra for the computations explained in this appendix.}
records and verifies Tietze transformations from $65$ generators and 240 formal relators to the $10$ generators and $25$ relators presentation stated in Proposition \zcref{algebraic_presentation}.
For the intermediate calculations, extend the word abbreviations and the relator numbering by
\begin{equation}\label{eq:three_conjugate_words}
\begin{aligned}
U_5&=g_9g_{10}g_9^{-1},&
q_{26}&=[g_{10},g_7^{-1}g_6g_7],\\
U_6&=g_8^{-1}U_2g_6U_2^{-1}g_8,&
q_{27}&=[U_5,U_6],\\
U_7&=g_8^{-1}U_2U_5U_2^{-1}g_8,&
q_{28}&=[U_7,g_6].
\end{aligned}
\end{equation}
The $U_i$ abbreviate words, without adding generators.  
The relators $q_{26},q_{27},q_{28}$ occur only in the intermediate presentation.
We list the final relators $q_1,\ldots,q_{25}$ in \zcref{relators_list}.
The intermediate target, in the basis \zcref{eq:newten}, is
\begin{equation}\label{eq:normalform28}
\left\langle g_1,\ldots,g_{10}\ \middle|\
q_1=\cdots=q_{28}=1\right\rangle .
\end{equation}

All input words and finite certificates are available in the
computational supplement linked under ``Notes \& Files'' on the
author's papers page:
\begin{center}
\url{https://sites.google.com/view/yukimatsubara/papers}.
\end{center}
The link leads to the Google Drive archive. For this release,
download \path{wobbly_n2_calculations_2026_09_25.zip}, extract it,
and run the following commands from the resulting
\path{wobbly_n2_calculations/} directory:
\begin{verbatim}
python3 -B verify_readable.py --audit
python3 -B verify_independence.py
\end{verbatim}
These commands require Python 3.9 or later and only its standard
library.
You can run them without \texttt{-O}, \texttt{-OO}, or
\texttt{PYTHONOPTIMIZE}.
The first reconstructs all $16\cdot15=240$ Schreier rewrites,
verifies the generator eliminations and reversible relation
replacements, derives the action table, checks the three final
conjugacy identities, and verifies the twenty-eight- and
twenty-five-relator presentations.
The \texttt{--audit} option independently checks every commutation
normalization using allowed adjacent swaps and free cancellations.
The second checks the exponent sums and primitive independence of
the twenty-five quadratic classes using an integral $25\times25$
minor of determinant $-1$.
The first command generates \path{proof_normal_form.json} in the
extracted directory. It records the precise intermediate
presentation \zcref{eq:normalform28} and the final presentation.
The second generates \path{verification/block_minor.json} and
\path{verification/quadratic_coefficients.csv}.
These generated files are not included in the archive, and the
extracted directory must be writable.

\subsection{Eliminating the $a$-type Schreier generators}\label{elimination}
Put $\sigma :=\sigma_1$ and $m_I=\sigma_{I,1}=a_I\sigma a_I^{-1}$ for $I\subseteq\{1,2,3,4\}$.
There are $16$ meridians $m_I$ and $49$ non-identity transition generators $a_{I,j}$.
For a fixed $j$, let
\begin{equation*}
    J_0=I\cap\{1,\ldots,j-1\},\qquad
J_1=J_0\cup\{j\},\qquad
\epsilon=\begin{cases}1&j\in I,\\0&j\notin I.\end{cases}
\end{equation*}
All the transition generators can be eliminated using the common formula
\begin{equation}\label{eq:transition_formula}
a_{I,j}
=m_I^{-1}m_{J_\epsilon}m_{J_{1-\epsilon}}^{-1}
 m_{I\triangle\{j\}}.
\end{equation}
Indeed, if $j\notin I$, write $a_I=a_{J_0}B$, where every index in $B$
is greater than $j$.  Relation~\eqref{rel:M3} says that $B$ commutes
with $\sigma a_j\sigma^{-1}$.  Expanding the right-hand side
of~\eqref{eq:transition_formula} therefore gives
$a_Ia_j a_{I\triangle\{j\}}^{-1}$.  The case $j\in I$ follows by
inverting the corresponding formula and using $a_j^2=1$.

Next we eliminate, in the order,
\[
m_{124},\quad m_{134},\quad m_{234},\quad
m_{1234},\quad m_{123},\quad m_{34}.
\]
Here, for example, $m_{124}$ means $m_{\{1,2,4\}}$.
Each defining relation contains the meridian to be eliminated exactly
once.  The surviving generators are
\[
(\widetilde g_1,\ldots,\widetilde g_{10})
=(m_\varnothing,m_1,m_2,m_{12},m_3,m_{13},m_{23},m_4,m_{14},m_{24}).
\]
After deleting empty relations and identifying cyclic conjugates and
inverses, $41$ relations remain.  The defining words for all $55$ eliminations, together with their order, are recorded in
\path{data/eliminate55.json}.  Each entry consists of the
eliminated generator, its replacement word, and its defining relator.
Substitution freely reduces the last word to the identity.

\subsection{Separating the action equations from the other relations}
At each step, a retained relator distinct from the one being replaced is cyclically permuted and possibly inverted to the form $AB$, and is used to replace a cyclically written relator $AX$ by $B^{-1}X$.
The $176$ recorded replacements, with free and cyclic reduction and identification up to cyclic permutation and inversion, remove twelve
duplicate relators and no empty relators, leaving a list $\rho_1,\ldots,\rho_{29}$.
The unnormalized words and the exact replacement identities are supplied in \path{relators_29_25.json} and \path{normalization.json}, respectively, in \path{data/}. The preceding replacements are recorded in \path{data/shorten176.json}.\\ 

In this appendix, $\rho_i$ always denotes the normalized word with integer key $i$ in the \texttt{normalized} entry of \path{normalization.json}.  
The normalization preserves the indices.
Each product identity uses its replaced relator exactly once, so the replacement is reversible while retaining all other relations.

Temporarily, we replace $\widetilde g_8$ by the free generator
\[
T=\widetilde g_2\widetilde g_3\widetilde g_5\widetilde g_8.
\]
Thus $\widetilde g_8=
\widetilde g_5^{-1}\widetilde g_3^{-1}\widetilde g_2^{-1}T$,
and $T=g_1g_8$ after \zcref{eq:newten}.

The action equations are obtained in the following order.
\begin{center}
\small
\renewcommand{\arraystretch}{1.3}
\begin{tabular}{c|p{0.55\textwidth}|l}
Number & Equations solved & Source relations\\\hline
16 & Eleven fixed images and five direct conjugations &
$\rho_i$, $i\in D$\\
2 & $\displaystyle
\widetilde g_5\widetilde g_j\widetilde g_5^{-1}
=(\widetilde g_2\widetilde g_3)^{-1}
T\widetilde g_jT^{-1}(\widetilde g_2\widetilde g_3)$,
$j=9,10$ & $\rho_3,\rho_1$\\
3 & $\widetilde g_2^{-1}T\widetilde g_2$,\quad
$\widetilde g_3^{-1}T\widetilde g_3$,\quad
$\widetilde g_5T\widetilde g_5^{-1}$ &
$\rho_{29},\rho_{25},\rho_{24}$
\end{tabular}
\end{center}
Here
\begin{equation*}
   D=\{2,5,7,8,9,10,11,12,13,14,15,17,18,19,21,22\}. 
\end{equation*}
In the first block, the fixed images are those of
\[
\begin{array}{c|l}
\text{conjugating generator}&\text{fixed generators}\\\hline
\widetilde g_2&\widetilde g_1,\widetilde g_4,\widetilde g_6,\widetilde g_9\\
\widetilde g_3&\widetilde g_1,\widetilde g_4,\widetilde g_7,\widetilde g_{10}\\
\widetilde g_5&\widetilde g_1,\widetilde g_6,\widetilde g_7 .
\end{array}
\]

The five other images in that block have index pairs $(2,7),(2,10),(3,6),(3,9),(5,4)$.
In the last block, each unresolved conjugate occurs exactly once after substituting the previously solved images.  
Solving for it and its inverse is therefore triangular. 
The two symbolic images in the middle block are evaluated as soon as the first two images of $T$ are known.
Changing basis by \zcref{eq:newten} and retaining $q_{23},q_{22},q_{26}$ during simplification gives precisely \zcref{tab:actions}.
Both compositions of every resulting action with its inverse are checked as free-group substitutions.

It remains to account for eight relations.  
Put
\[
\widehat q_{24}=g_9q_{24}g_9^{-1}=[U_5,g_6],\qquad
\widehat q_{25}=g_4^{-1}q_{25}g_4.
\]
After collection with the
action equations, they have the following equivalent defining forms:
\begin{equation}\label{eq:residual_eight}
\renewcommand{\arraystretch}{1.2}
\begin{array}{c|l@{\qquad}c|l}
\text{Source}&\text{Form}&\text{Source}&\text{Form}\\\hline
\rho_4&q_{22}&\rho_6&q_{23}\\
\rho_{16}&q_{24}&\rho_{23}&q_{25}\\
\rho_{20}&q_{26}&\rho_{26}&q_{23}\\
\rho_{27}&[U_5^{-1},g_6U_6^{-1}]&
\rho_{28}&U_7^{-1}\widehat q_{25}g_6^{-1}U_7g_6
\end{array}
\end{equation}

Entries in this table specify equivalent relations, allowing conjugation or inversion of an individual relator.
They are not asserted to be identical free words.
More precisely, if $\overline\rho_i$ denotes the collected word, then
\[
g_2\overline\rho_{27}g_2^{-1}
=U_1[U_5^{-1},g_6U_6^{-1}]U_1^{-1}
\]
is an exact identity after substituting the action equations.
For $\rho_{28}$, first normalize using $q_{23}$, and then conjugate by $g_3^{-1}$.
Its replacement by $U_7^{-1}\widehat q_{25}g_6^{-1}U_7g_6$ is valid \emph{modulo the retained relations $q_{23},q_{22},q_{26}$ and
$g_3^{-1}q_{23}g_3$}.
The file \path{data/rho28_factorization.json} gives the explicit product of conjugates proving this statement.
Its product identity is checked modulo $q_{23},q_{22}$.
Removing only the factors conjugate to $q_{26}$ and $g_3^{-1}q_{23}g_3$ leaves one conjugate of $U_7^{-1}\widehat q_{25}g_6^{-1}U_7g_6$.
In particular, this step does not assume the relation $q_{28}=1$ that it will produce.
We also recheck the relations used to solve the actions after simplifying to \zcref{tab:actions}.
Thus no defining relation is lost by simplifying an image.

\subsection{Completing the Tietze transformations}
The following two identities hold in the free group:
\begin{align*}
[U_5^{-1},g_6U_6^{-1}]
&=(U_5^{-1}\widehat q_{24}^{-1}U_5)
  (g_6U_6^{-1}U_5^{-1})q_{27}(g_6U_6^{-1}U_5^{-1})^{-1},\\
U_7^{-1}\widehat q_{25}g_6^{-1}U_7g_6
&=(U_7^{-1}\widehat q_{25}U_7)
  (g_6^{-1}U_7^{-1})q_{28}(g_6^{-1}U_7^{-1})^{-1}.
\end{align*}
Since $\widehat q_{24}$ and $\widehat q_{25}$ are conjugates of $q_{24}$ and $q_{25}$, the retained relations $q_{24}=q_{25}=1$ allow the last two entries of \zcref{eq:residual_eight} to be replaced reversibly by $q_{27}=q_{28}=1$.
Deleting the duplicate $q_{23}$ gives exactly \zcref{eq:normalform28}.

Finally, the equations of \zcref{tab:actions} alone give
\begin{equation}\label{eq:three_conjugates}
\begin{split}
q_{26}&=(U_1g_3^{-1}g_2^{-1}g_9)\,q_{24}\,(U_1g_3^{-1}g_2^{-1}g_9)^{-1},\\
q_{27}&=(g_8^{-1}g_5g_9)\,q_{24}\,(g_8^{-1}g_5g_9)^{-1},\\
q_{28}&=(U_2g_5^{-1}g_9)\,q_{24}\,(U_2g_5^{-1}g_9)^{-1}.
\end{split}
\end{equation}
Since $q_{24}=1$ is retained, $q_{26},q_{27},q_{28}$ are redundant and may be deleted.
This completes the Tietze transformations to $\langle g_1,\ldots,g_{10}\mid q_1,\ldots,q_{25}\rangle$.

\section{The complete twenty-five relators}\label{relators_list}
We use the generators $g_1, \dots, g_{10}$ of \zcref{eq:newten}.
The following are the ordered relators of Proposition \zcref{algebraic_presentation}.
Their order agrees with \path{RELATORS} in \path{presentation.py} in the computational supplement.
Here all word abbreviations have been expanded.
\allowdisplaybreaks

\begin{align*}
q_{1} &= g_{2} g_{1} g_{2}^{-1} g_{1}^{-1}\\[2pt]
q_{2} &= g_{2} g_{4} g_{2}^{-1} g_{4}^{-1}\\[2pt]
q_{3} &= g_{2} g_{6} g_{2}^{-1} g_{6}^{-1}\\[2pt]
q_{4} &= g_{2} g_{7} g_{2}^{-1} g_{1} g_{4} g_{6} g_{7}^{-1} g_{6}^{-1} g_{4}^{-1} g_{1}^{-1}\\[2pt]
q_{5} &= g_{2} g_{8} g_{2}^{-1} g_{4} g_{6} g_{9} g_{8}^{-1} g_{9}^{-1} g_{6}^{-1} g_{4}^{-1}\\[2pt]
q_{6} &= g_{2} g_{9} g_{2}^{-1} g_{9}^{-1}\\[2pt]
q_{7} &= g_{2} g_{10} g_{2}^{-1} g_{1} g_{4} g_{9} g_{10}^{-1} g_{9}^{-1} g_{4}^{-1} g_{1}^{-1}
\end{align*}

\begin{align*}
q_{8} &= g_{3} g_{1} g_{3}^{-1} g_{1}^{-1}\\[2pt]
q_{9} &= g_{3} g_{4} g_{3}^{-1} g_{4}^{-1}\\[2pt]
q_{10} &= g_{3} g_{6} g_{3}^{-1} g_{7} g_{1} g_{4} g_{6}^{-1} g_{4}^{-1} g_{1}^{-1} g_{7}^{-1}\\[2pt]
q_{11} &= g_{3} g_{7} g_{3}^{-1} g_{7}^{-1}\\[2pt]
q_{12} &= g_{3} g_{8} g_{3}^{-1} g_{4} g_{7} g_{10} g_{1} g_{8}^{-1} g_{1}^{-1} g_{10}^{-1}\\*
&\quad g_{7}^{-1} g_{4}^{-1}\\[2pt]
q_{13} &= g_{3} g_{9} g_{3}^{-1} g_{10} g_{1} g_{4} g_{9}^{-1} g_{4}^{-1} g_{1}^{-1} g_{10}^{-1}\\[2pt]
q_{14} &= g_{3} g_{10} g_{3}^{-1} g_{10}^{-1}
\end{align*}

\begin{align*}
q_{15} &= g_{5} g_{1} g_{5}^{-1} g_{1}^{-1}\\[2pt]
q_{16} &= g_{5} g_{4} g_{5}^{-1} g_{4} g_{6} g_{7} g_{4}^{-1} g_{7}^{-1} g_{6}^{-1} g_{4}^{-1}\\[2pt]
q_{17} &= g_{5} g_{6} g_{5}^{-1} g_{4} g_{6} g_{7} g_{6}^{-1} g_{7}^{-1} g_{6}^{-1} g_{4}^{-1}\\[2pt]
q_{18} &= g_{5} g_{7} g_{5}^{-1} g_{4} g_{6} g_{7}^{-1} g_{6}^{-1} g_{4}^{-1}\\[2pt]
q_{19} &= g_{5} g_{8} g_{5}^{-1} g_{4} g_{6} g_{7} g_{8}^{-1} g_{7}^{-1} g_{6}^{-1} g_{4}^{-1}\\[2pt]
q_{20} &= g_{5} g_{9} g_{5}^{-1} g_{8} g_{9}^{-1} g_{8}^{-1}\\[2pt]
q_{21} &= g_{5} g_{10} g_{5}^{-1} g_{8} g_{10}^{-1} g_{8}^{-1}
\end{align*}

\begin{align*}
q_{22} &= g_{9} g_{7} g_{9}^{-1} g_{7}^{-1}\\[2pt]
q_{23} &= g_{1} g_{8} g_{1}^{-1} g_{8}^{-1}\\[2pt]
q_{24} &= g_{10} g_{9}^{-1} g_{6} g_{9} g_{10}^{-1} g_{9}^{-1} g_{6}^{-1} g_{9}\\[2pt]
q_{25} &= g_{4} g_{8}^{-1} g_{4} g_{6} g_{7} g_{9} g_{10} g_{1} g_{4}^{-1} g_{1}^{-1}\\
&\quad g_{10}^{-1} g_{9}^{-1} g_{7}^{-1} g_{6}^{-1} g_{4}^{-1} g_{8}.
\end{align*}

\section{The algebraic calculation for $n = 3$}\label{appndix_C}
This appendix\footnote{We used Codex for the computations explained in this appendix.}
records the Tietze transformations from the $193$-generator, $864$-relator Reidemeister--Schreier presentation of Theorem \zcref{Explicit_description_of_fundamental_group} to a presentation with $4$ generators and $78$ relators.
The scripts \path{n3_model.py} and \path{independent_model.py} independently reconstruct the input from \zcref{presentation}.
\path{verify_input.py} checks their agreement and the abelianization $\bZ/8\bZ$.

The file \path{tietze_phases.json} lists $26$ certificate phases.
They record generator eliminations using single-occurrence defining relators, replacements using retained relators, and elementary Nielsen changes of generators.
Further deletions and shortenings are certified by reversible free-group rewriting using explicitly proved consequences of retained relators.
The program \path{verify_tietze.py} replays the sequence and checks that, in every such consequence rewrite, all supporting relators remain present and exclude the target relator.
It compares the final $78$ relators and the definitions of the four generators in the ambient group with \path{presentation.json}.
The final presentation is displayed in \path{presentation_g.txt}.

The distribution page for the computational supplement is the author's papers page:
\begin{center}
\url{https://sites.google.com/view/yukimatsubara/papers}.
\end{center}
The download link is listed under ``Notes \& Files'' and leads to Google Drive.
The version used here is supplied as \path{wobbly_n3_tietze_2026_09_25.zip}, with SHA-256
\begin{center}
\small\ttfamily
af3a5eb9584f6e4f6eab4f0d7fc81c0d5f2fdb4d028a7345440bbaeb276bad9c
\end{center}
Extract it and run the following command from \path{wobbly_n3_tietze/}:
\begin{verbatim}
python3 -B verify_tietze.py
\end{verbatim}
Python 3.9 or later and its standard library suffice; run without \texttt{-O}, \texttt{-OO}, or \texttt{PYTHONOPTIMIZE}.
The command must exit with status $0$ and report \texttt{PASS}.
It generates reports in \path{verification/}, so the extracted directory must be writable.
The supplied \path{SHA256SUMS} excludes these generated reports.

\bibliographystyle{alpha}
\bibliography{references.bib}

\Address

\end{document}